\documentclass[12pt,a4paper]{article}
\usepackage{caption}
\usepackage{subcaption}
 \usepackage{amssymb,amsmath,amsfonts,amsthm}
  \usepackage{graphics,graphicx,color,transparent}
 \advance\hoffset by 5mm

\numberwithin{equation}{section}

\newcommand{\R}{{\mathbb R}}
\newcommand{\mcU}{{\mathcal U}}
\newcommand{\Sph}{{\mathbb S}}
\newcommand{\mcC}{{\mathcal C}}
\newcommand{\mcS}{{\mathcal S}}
\newcommand{\mcG}{{\mathbb G}}
\newcommand{\To}{{\mathbb T}^2}
\newcommand{\dist}{{\rm dist}\,}
\newcommand{\defeq}{\stackrel{\rm{def}}{=}}

\newcommand{\eps}{\varepsilon}
\newcommand{\om}{\Omega}
\newcommand{\supp}{{\rm supp}\,}

\def\be{\begin{equation}}
\def\ee{\end{equation}}

\newcounter{marnote}

 \title{Partial regularity and smooth topology-preserving approximations of
rough domains}
 \date{\today}
\author{John M. Ball\thanks{Oxford Centre for Nonlinear PDE, Mathematical Institute, University of Oxford, Andrew Wiles Building, Radcliffe Observatory Quarter, Woodstock Road, Oxford OX2 6GG, U.K. }\,  and Arghir Zarnescu\thanks{ IKERBASQUE, Basque Foundation for Science, Maria Diaz de Haro 3,
48013, Bilbao, Bizkaia, Spain}\,\,\thanks{BCAM, Basque Center for Applied Mathematics, Mazarredo 14, E48009 Bilbao, Basque Country, Spain}\,\,\thanks{ ``Simion Stoilow" Institute of the Romanian Academy, 21 Calea Grivi\c{t}ei Street,  010702 Bucharest, Romania}}
\newtheorem{definition}{Definition}[section]
\newtheorem{remark}{Remark}[section]
 \newtheorem{lemma} {Lemma}[section]
 \newtheorem{proposition}{Proposition}[section]
 \newtheorem{theorem} {Theorem}[section]
 
\newtheorem{example}{Example}[section]

\begin{document}
\maketitle
\centerline {{\it In memoriam} J. Bryce McLeod}
\begin{abstract}
\noindent For a bounded domain $\Omega\subset\R^m, m\geq 2,$   of class $C^0$,   the properties are studied of fields of `good directions', that is the directions  with respect to which $\partial\Omega$ can be locally represented   as the graph of a continuous function. For any such domain there is a canonical smooth field of good directions defined in a suitable neighbourhood of $\partial\Omega$, in terms of which   a corresponding flow can be defined. Using this flow it is shown that $\Omega$ can be approximated from the inside and the outside by diffeomorphic domains of class $C^\infty$. Whether or not the image of a  general continuous field of good directions (pseudonormals) defined on $\partial\Omega$ is the whole of $\mathbb{S}^{m-1}$ is shown to depend on  the topology of $\Omega$. These considerations are  used to  prove that if $m=2,3$, or if $\Omega$ has nonzero Euler characteristic, there is a  point $P\in\partial\Omega$  in the neighbourhood of which $\partial\Omega$ is Lipschitz. The results provide new information even for more regular domains, with Lipschitz or smooth boundaries.
\end{abstract}
 \section{Introduction}  
In this paper we study bounded domains $\Omega\subset\R^m$, $m>1$, of class $C^0$, showing that they can be approximated from the inside and outside by bounded domains $\Omega_\eps$ of class $C^\infty$ that are diffeomorphic to $\Omega$, and such that $\bar\Omega_\eps$ are homeomorphic to $\bar\Omega$. Thus the approximating smooth domains preserve topological properties of the rough domain; in particular, for instance, a simply-connected bounded domain of class $C^0$ can be approximated from the inside and outside by smooth simply-connected domains.  The method of approximation uses a construction of smooth fields of `good directions', that is  directions with respect to which the boundary $\partial\Omega$ can be locally represented as the graph of a continuous function, and a corresponding flow. We analyze topological properties of fields of good directions, and exploit them to  study partial regularity of the boundary of such domains.

 Domains of class $C^0$ represent  one of the largest class of domains relevant to the analysis of partial differential equations and their numerical approximation, and have been widely studied,
see for  instance \cite{necas-paper,necas-book,Grisvard,mazya}.  They are in particular relevant for  physical applications to  non-Lipschitz domains arising  in optimal design, free-boundary problems (e.g. for two-phase flow) and fracture (see for instance \cite{coutand, shape, friedman,hadzic,pironneau}). Such rough domains  also appear in the theoretical study of parabolic   equations through the use of H\"older-continuous space-time domains  \cite{dindos,lieberman-book}  or  elliptic problems in domains with point singularities \cite{mazya2}.
Other applications include problems in which the domain is an unknown (such as optimal design), for which the flow of good directions might be used to construct domain variations, and various problems in partial differential equations \cite{Henry} and potential theory \cite{Staubach}. Nevertheless there is a 
 a lack of tools for treating general domains of class $C^0$ and they are often treated on a case-by-case basis, under various additional simplifying assumptions (see for instance \cite{mazya2,mazya3}). 
One aim of our study is to provide a set of versatile tools for dealing with such domains, without the need of {\it ad hoc} methods adjusted to specific cases.

We recall that $\Omega\subset\R^m$, $m>1$ is {\it a domain of class $C^0$} (respectively {\it of class $C^r$}, $r=1,2,\ldots,\infty$, Lipschitz) if $\Omega$ is a connected open set such that for any
point $P$ belonging to the boundary $\partial\Omega$ there exist a $\delta>0$ and an
orthonormal coordinate system $Y\stackrel{\rm def}{=}(y',y_m)=(y_1,y_2,\dots,y_m)$  with origin at $P$,
together with a continuous (respectively $C^r$, Lipschitz) function
$f:\R^{m-1}\to\R$, such that
\be
\Omega\cap B(P,\delta)=\{y\in\R^m:\, y_m>f(y'), |y|<\delta\}
\label{u}
\ee from which it follows that $\partial\Omega\cap B(P,\delta)=\{y\in\R^m: y_m=f(y'), |y|<\delta\}$ and $f(0)=0$.
We call the unit vector $n(P)=e_m(P)$ in the $y_m-$direction for this coordinate system a {\it pseudonormal} at $P$. More generally, if $P\in\R^m$ is not necessarily a boundary point, but is such that for the coordinate system $Y$ with origin at $P$ we have that \eqref{u} holds with $\partial\Omega\cap B(P,\delta)$ nonempty, we call the corresponding unit vector $n(P)$ {\it a good direction at} $P$.
We show (Lemma \ref{geoconv}) that the set of good directions at a point $P$ is a (geodesically) convex subset of the unit sphere $\Sph^{m-1}$. Using a partition of unity we deduce (Proposition \ref{smapp}) that for a bounded domain of class $C^0$ there exists a field of good directions $G(P)$,  that we call {\it canonical}, depending {\it smoothly} on $P$.

Although good directions and pseudonormals are defined locally, their properties depend on the topology of  the domain. If $\Omega$ is a bounded domain of class $C^1$ then the negative Gauss map $\nu_{\partial\Omega}:\partial\Omega\rightarrow \Sph^{m-1}$, defined by $\nu_{\partial\Omega}(P)=$ the inward normal to $\partial\Omega$ at $P$, is surjective. We show (Theorem \ref{prop:surjectivity}) that the same is true for an arbitrary continuous field of pseudonormals if $m=2$ or if $m\geq 3$ and $\Omega$ has nonzero Euler characteristic. However if $\Omega\subset\R^3$ is a standard solid torus in $\R^3$ then, using an observation of Lackenby \cite{lackenby}, we show (Proposition \ref{prop:goodtorus}) 
that there is a continuous field of pseudonormals with  image contained in an arbitrarily small neighbourhood of the great circle in $S^2$ perpendicular to the axis of cylindrical symmetry of the torus.

In our approximation result (Theorem \ref{homapprox}) the approximating domains are given by
\be\label{intro1}
\Omega_\eps\stackrel{\rm def}{=}\{x\in\R^m:\rho(x)>\eps\}, \;\; 0<|\eps|<\eps_0,
\ee
where $\rho$ is   a regularized signed distance to the boundary. In general these domains would provide, for suitably small $\varepsilon_0$ and almost any $\varepsilon\in (-\varepsilon_0,\varepsilon_0)$, an interior and exterior approximation for an {\it arbitrary} open set $\Omega$ (see Remark~\ref{remark:arbitrary omega}). However if one requires that the approximating domains are in the same diffeomorphism (or homeomorphism up to the boundary) class as the initial domain one needs to impose some restrictions on $\Omega$ (see Examples~\ref{exteriorex}, \ref{interiorex}). The bounded domains of class $C^0$ form a large class of domains for which this type of approximation is possible.(In fact Theorem \ref{homapprox} holds for trivially for the larger class of domains that are the image of a bounded domain of class $C^0$ under a suitable diffeomorphism; see Remark \ref{genapprox}.)

In our paper we use the regularized signed distance to the boundary defined by Lieberman \cite{lieberman}. This regularized distance has the property that $\nabla\rho(P)\cdot G(P)>0$ for a canonical good direction $G(P)$ in a suitable neighbourhood of $\partial\Omega$. This enables us to use the {\it flow of canonical good directions} $S(t)x_0$ defined as the solution of 
$$
  \begin{array}{ll}
          \dot x(t)=G(x(t))&\textrm{ for $t\in\R$},\\
          x(0)=x_0,
          \end{array}  
$$
suitably extended so as to be globally defined, to provide the deformation showing that  $\Omega$ and $\Omega_\eps$ are $C^\infty$-diffeomorphic and their closures  are homeomorphic.

A  surprising by-product of our study of the topology of the set $\Omega$ and that of the good directions is  the fact that in $\R^m$ for $m=2,3$   any bounded domain of class $C^0$ must necessarily have portions of the boundary with better regularity, namely Lipschitz regularity (Theorems \ref{partial2Da}, \ref{partial3D}). This is true for arbitrary $m$ if $\Omega$ has nonzero Euler characteristic (Theorem \ref{thm:zeroEulerchar}). However it is  in general false for  {\it unbounded} domains of class $C^0$ (Remark \ref{unbounded}).

\medskip
Let us now mention some related literature. Further results on domains of class $C^0$ and their properties,   in particular their geometric characterization as domains with the segment property\footnote{i.e. for each point $P\in\partial\Omega$ there exists a neighbourhood $U(P)$ in $\R^m$ and a non-zero vector $b(P)\in\R^m$ such that $x+tb\in\Omega,\mbox{ for all } x\in\overline{\Omega}\cap U(P), 0<t<1$.},  are available in Fraenkel  \cite{fraenkelc0}. The relation between domains of class $C^0$ and their closure is addressed in Grisvard \cite[Theorem 1.2.1.5]{Grisvard}. Also in \cite[Corollary 1.2.2.3]{Grisvard},    it is noted that a bounded open convex set necessarily has Lipschitz boundary.
Somewhat related domains are the `cloudy manifolds' defined by Kleiner \& Lott  \cite{KleinerLott} which are subsets of an Euclidean space with the property that near each point they look ``coarsely close'' to an affine subspace of the Euclidean space. It was shown in \cite{KleinerLott} that any cloudy $k$-manifold can be
well interpolated by a smooth $k$-dimensional submanifold of the Euclidean space. 
Another strand of research that can be compared with our partial regularity result for bounded domains of class $C^0$ is that described by Jones, Katz \& Vargas \cite{jones}, in which the authors prove and generalize a conjecture of Semmes \cite{semmes}   that for a bounded open set   $\Omega\subset\R^m$ with $\mathcal H^{m-1}(\partial\Omega)=M<\infty$ there exist $\varepsilon>0$ and a Lipschitz graph $\Gamma$ with $\mathcal H^{m-1}(\Gamma\cap\partial\Omega)\geq\varepsilon$.

  Close in spirit to our work is that of Verchota (\cite[Appendix]{verchota-thesis}, \cite[Proposition 1.12]{verchota-JFA}), who studies bounded Lipschitz domains $\Omega$ and constructs  smooth approximating domains whose boundaries are shown to be homeomorphic with $\partial\Omega$ using a smooth flow.  

In Hofmann, Mitrea \& Taylor \cite{hofmann} a definition is given of a continuous vector field transversal to the boundary of an open set with locally finite perimeter, in which it is required that the inner product of the vector field with the normal is bounded away from zero. For the case of   bounded domains of class $C^0$ (which need not have finite perimeter) this is a similar but stronger requirement than being a continuous field of good directions. A result  \cite[Proposition 2.2] {hofmann} analogous to our Proposition \ref{smapp} is then proved giving conditions under which the existence of a continuous  locally tranversal field implies the existence of a global smooth transversal field.

  The closest to our work seems to be that  of Iwaniec \& Onninen \cite{iwaniec}. There they note, as we do in Proposition~\ref{lemma:nonzero}, that the distance is a monotone function along a good direction, and they use a type of canonical field of good directions and its flow to construct approximating domains similarly to our construction in Section~\ref{section:flow}.  

Finally, in a more general framework related types of questions were addressed in the papers of Cairns \cite{cairns}, Whitehead \cite{whitehead} and Pugh  \cite{Pugh}, which study conditions under which a  topological manifold can be smoothed: that is when  its (maximal) topological atlas contains a smooth subatlas. Indeed, it seems likely that the techniques of smoothing of manifolds could lead to inner and outer approximations by domains of class $C^\infty$ for the wider class of bounded domains $\Omega\subset\R^m$, $m\neq 4,5$, having boundaries that are locally flat (that is, $\bar\Omega$ and $\Omega^c$ are topological manifolds with boundary), though without the same kind of semi-explicit representation as in Theorem \ref{homapprox} of the approximating domains in terms of a regularized distance. (A brief comparison between such domains and those of class $C^0$  can be found in \cite{Grisvard}, p.$5-10$.) This is discussed in more detail in Remark \ref{topman1}.

Although the questions addressed in this paper seem natural, we are not aware of any other work that studies the properties of good directions and their spatial variation, or of related approximation results preserving topological properties.  We were motivated to consider these questions because in an analysis \cite{ballzarnescu} of orientability for liquid crystals we wanted to use a result of Pakzad \& Rivi\`ere \cite{pakzadriviere}, which assumed $\Omega$ to be of class $C^\infty$ and simply-connected, for more general simply-connected domains (though recent work of Bedford \cite{bedford} gives a way of proving the desired orientability without approximation of the domain).

\section{Good directions}
\label{good}

\par

\smallskip\begin{definition}\label{gooddir} Let $\Omega\subset\R^m$ be a domain of class $C^0$.
For a point $P\in\R^m$ we define a {\rm good direction at $P$, with respect to a ball $B(P,\delta)$}, $\delta>0$, with $B(P,\delta)\cap\partial\Omega\not=\emptyset$ to  be a vector $n\in\Sph^{m-1}$ such that  there is an
orthonormal coordinate system $Y=(y',y_m)=(y_1,y_2,\dots,y_m)$  with origin at  the point $P$ and such that $n=e_m$
is the unit vector in the $y_m$ direction,  together with a continuous function
$f:\R^{m-1}\to\R$ (depending on $P$ and $n$ and $\delta$), such that
\be
\Omega\cap B(P,\delta)=\{y\in\R^m:\,y_m>f(y'),\,|y|<\delta\}.
\label{u+}
\ee
 \par We say that $n$ is a {\rm good direction at $P$} if it is a good direction with respect to some ball $B(P,\delta)$ with $B(P,\delta)\cap\partial\Omega\not=0$.
\par If $P\in\partial\Omega$ then a good direction $n$  at $P$ is called a {\rm pseudonormal} at  $P$.

\end{definition}

\begin{remark} {\rm A good direction at a point need not be
unique but for a bounded domain of class $C^0$ there always exists at least one  for each point in a  (small enough)  neighbourhood of
$\partial\Omega$. Note also that if both $n$ and $\bar n$ are good directions at $P$ then we can choose $\delta=\min\{\delta(P,n),\delta(P,\bar n)\}$ so that the corresponding two representations \eqref{u+} hold for $\delta$. However a possible choice of $\delta(P,n)$ may not be a possible choice of $\delta(P,\bar n)$.}
\end{remark}

\begin{remark}\label{cusp} {\rm If one has for instance a domain in $\R^2$ such that part of its boundary can be locally represented as $\{(x,f(x)), x\in (-1,1)\}$ with $f(x)=\sqrt{|x|}$ then there is only one good direction at the point $(0,0)$ namely $(0,1)\in\Sph^1$. This suggests that there exists a connection between the uniqueness of a good direction and the regularity of the boundary, and this topic will be explored in detail in the last section.}
\end{remark}

 We refer to a  subset $S$ of a Riemannian manifold $M$ as {\it geodesically convex} if given any two points $P,Q\in S$ there is a unique shortest curve (minimal geodesic) in $M$ joining $P$ and $Q$, and this curve lies in $S$ (there are several differing definitions in the literature). The following lemma asserts that the set of good directions at any given point is a geodesically convex subset of $\mathbb{S}^{m-1}$.  For a closely related result see Wilson \cite[Lemma 4.2]{wesleywilson}.
\begin{lemma} 
\label{geoconv}
Let $\Omega\subset\R^m$ be a $($possibly unbounded$)$ domain of class $C^0$. If $p,q\in\Sph^{m-1}$ are  good directions at a point
$P\in\R^m$ with respect to the ball $B(P,\delta)$ then $p\ne -q$ and for any
$\lambda\in (0,1)$ the vector
$\frac{\lambda p+(1-\lambda) q}{|\lambda p+(1-\lambda)q|}$ is a
good direction at $P$ with respect to the ball $B(P,\delta)$. \label{lemma:two}
\end{lemma}

\begin{proof}
Let $p,q$ be good directions at $P$, and $0<\lambda<1$. We can assume that $P=0$. Let $B\stackrel{\rm def}{=}B(0,\delta)$. If $p=q$ there is nothing to prove. So assume $p\neq q$. Then Definition \ref{gooddir} implies that $p\neq -q$, and so $\lambda p+(1-\lambda)q\neq 0$. Let $N=\frac{\lambda p+(1-\lambda)q}{|\lambda p+(1-\lambda)q|}$. We choose coordinates such that $e_m=N$. Take any $\xi\in B\backslash\Omega$. Then since $p$ and $q$ are good directions the intersections of $B$ with the open half-lines $\{\xi+tp:t<0\}$ and $\{\xi+tq:t<0\}$ lie in $\R^m\backslash\overline\Omega$. Similarly, if $\xi\in B\cap\overline\Omega$ the intersections of $B$ with the open half-lines $\{\xi+tp:t>0\}$ and $\{\xi+tq:t>0\}$  lie in $\Omega$.

\begin{figure}[h] \centering \def\svgwidth{300pt} 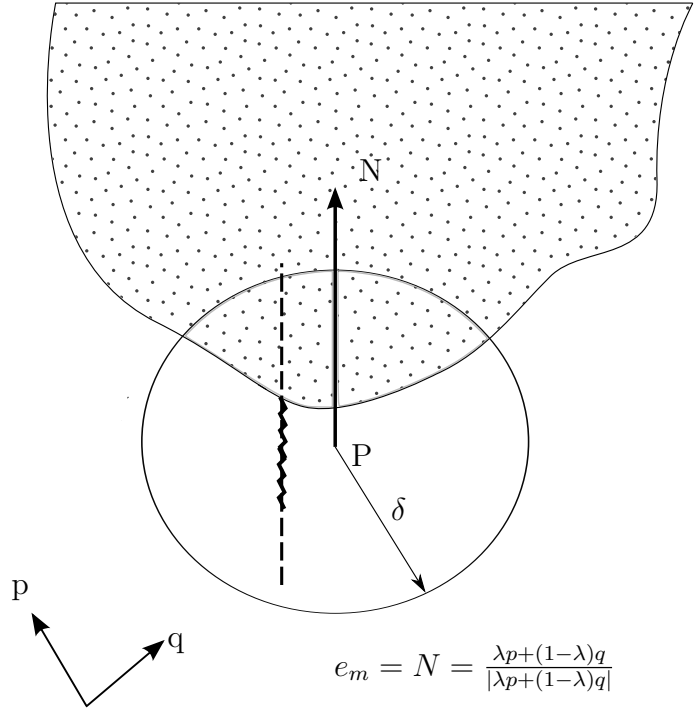
\caption{Geodesically convex combination of two good directions}
\end{figure}

Given any $\xi'\in B_{m-1}(0,\delta)=\{z\in\R^{m-1}:|z|<\delta\}$ define the line $L(\xi')=\{(\xi',t):t\in\R\}$, and let $S=\{\xi'\in B_{m-1}(0,\delta): L(\xi')\cap \partial\Omega\cap B\mbox{ is nonempty}\}$. We claim that if $\xi'\in S$ then $L(\xi')$ intersects $B\cap\partial\Omega$ in a unique  point $(\xi',t(\xi'))$, and that $\{(\xi',t):t>t(\xi')\}\cap B\subset\Omega$ and $\{(\xi',t):t<t(\xi')\}\cap B\subset\R^m\backslash\overline\Omega$. To prove the claim let $\xi=(\xi',t(\xi'))\in B\cap\partial\Omega$, and suppose $\xi-he_m\in B$ for some $h>0$. Then for some $\eps>0$, $\dist(\xi,\partial B)>\eps, \dist (\xi-he_m,\partial B)>\eps$. Choose a positive integer $k>\frac{h}{\eps |\lambda p+(1-\lambda)q|}$, and divide the interval $(0,h)$ into $k$  subintervals of length $h/k$. Then $\bar \xi=\xi-\frac{h}{k}\cdot\frac{(1-\lambda)q}{|\lambda p+(1-\lambda)q|}\in B$, and so $\bar \xi\in\R^m\backslash\overline\Omega$.   Thus $\xi-\frac{h}{k}e_m=\bar \xi-\frac{h}{k}\cdot\frac{\lambda p}{|\lambda p+(1-\lambda)q|}\in\R^m\backslash\overline\Omega$ and $\dist(\xi-\frac{h}{k}e_m,\partial\Omega)>\eps$. Proceeding inductively, after $k$ steps we find that $\xi-he_m\in\R^m\backslash\overline\Omega$, so that  $\{(\xi',t):t<t(\xi')\}\cap B\subset\R^m\backslash\overline\Omega$. It follows similarly that $\{(\xi',t):t>t(\xi')\}\cap B\subset\Omega$, establishing the claim.

  Now define
\begin{eqnarray}\label{fdef}
f(\xi')=\left\{\begin{array}{ll} 0&\mbox{if }|\xi'|=\delta\\ t(\xi')&\mbox{if } \xi'\in S\\
-\sqrt{\delta^2-|\xi'|^2}&\mbox{if }L(\xi')\cap B=L(\xi')\cap B\cap\Omega\\
\sqrt{\delta^2-|\xi'|^2}&\mbox{if }L(\xi')\cap B=L(\xi')\cap B\cap(\R^m\backslash\overline\Omega).
\end{array}\right.\end{eqnarray}

Note that $(\xi',f(\xi'))\in\bar B$ for all $\xi'\in B_{m-1}(0,\delta)$. Clearly $\Omega\cap B=\{(\xi',\xi_m)\in B:|\xi'|<\delta, \xi_m>f(\xi')\}$, and it remains to prove that $f$ is continuous, since then we can extend $f$ by zero for $|\xi'|>\delta$ to get a suitable continuous $f:\R^{m-1}\rightarrow \R$.
Let $\xi'^{(j)}\rightarrow \xi'$ in $\overline{B_{m-1}(0,\delta)}$. If $|\xi'|=\delta$ then $|(\xi'^{(j)},f(\xi'^{(j)}))|\leq\delta$ implies $|\xi'^{(j)}|^2+f^2(\xi'^{(j)})\leq\delta^2$ and so $f(\xi'^{(j)})\rightarrow 0=f(\xi')$.
If $\xi'\in S$ then $(\xi',f(\xi'))\in B\cap\partial\Omega$ and so for any sufficiently small $\eps>0$ the points $x_\eps^+=(\xi',f(\xi')+\eps)$ and $x_\eps^-=(\xi',f(\xi')-\eps)$ belong to $\Omega\cap B$ and to $(\R^m\backslash\overline\Omega)\cap B$ respectively. Since $\Omega$ and $\R^m\backslash\overline\Omega$ are open, there exists $\delta\in(0,\eps)$ such that $B(x_\eps^+,\delta)\subset \Omega\cap B$ and $B(x_\eps^-,\delta)\subset(\R^m\backslash \overline\Omega)\cap B$. Hence for sufficiently large $j$, the line $L(\xi'^{(j)})$ has points in both $B(x_\eps^+,\delta)$ and $B(x_\eps^-,\delta)$ and thus intersects $\partial\Omega$ in $B$ at the unique point $(\xi'^{(j)},f(\xi'^{(j)}))$, where $|f(\xi'^{(j)})-f(\xi')|<\eps$. Since $\eps$ is arbitrarily small, $f(\xi'^{(j)})\rightarrow f(\xi')$. Similarly, if $L(\xi')\cap B=L(\xi')\cap\Omega$ so that $f(\xi')=-\sqrt{\delta^2-|\xi'|^2}$, then for all sufficiently small $\eps>0$ there is a $\delta\in (0,\eps)$ such that the ball $B((\xi',f(\xi')+\eps), \delta)\subset\Omega\cap B$. Hence for all sufficiently large $j$ we have $f(\xi'^{(j)})\leq f(\xi')+\eps$, so that $f(\xi'^{(j)})\rightarrow f(\xi')$. The case when $L(\xi')\cap B=L(\xi')\cap B\cap(\R^m\backslash\overline\Omega)$ is handled in a similar way.
\end{proof}

\smallskip The above can be easily extended to the case
of an arbitrary number of  good directions:

\begin{lemma}\label{ggd}Let $k=1,2,\dots$ .  If $n_1,n_2,\dots, n_k\in\Sph^{m-1}$ are good
directions at a point
$P$ with respect to the ball $B(P,\delta)$  and
$0<\lambda_i<1,i=1,2,\dots,k$, with
$\Sigma_{i=1}^k\lambda_i=1$ then $\Sigma_{i=1}^k
\lambda_i n_i\neq 0$ and $\frac{\Sigma_{i=1}^k
\lambda_i n_i}{|\Sigma_{i=1}^k\lambda_i n_i|}$ is a good direction at $P$ with respect to $B(P,\delta)$.
 \label{lemma:multiple}
\end{lemma}

\begin{proof}  This follows easily from Lemma~\ref{lemma:two} by induction on $k$, noting that $\sum_{i=1}^k\lambda_in_i$ is parallel to
\begin{eqnarray*}\hspace{-2.8in}\mu\frac{\sum_{i=1}^{k-1}\lambda_in_i}{\left|\sum_{i=1}^{k-1}\lambda_in_i\right|}+(1-\mu)n_k,\mbox{ where    }
\mu=\frac{\left|\sum_{i=1}^{k-1}\lambda_i n_i\right|}{\lambda_k+\left|\sum_{i=1}^{k-1}\lambda_i n_i\right|}.
\end{eqnarray*}  \end{proof}

\par Despite the fact that the boundary is just of class $C^0$ we can  easily construct a smooth field of good directions in a neighbourhood of the boundary:

\begin{proposition} 
\label{smapp}
Let $\Omega\subset\R^m$ be a bounded, open set with
boundary of class $C^0$.  There exists a neighbourhood $U$ of
$\partial\Omega$ and  a smooth function
$G:U\to\Sph^{m-1}$ so that for each $P\in U$ the
unit vector $G(P)$ is a good direction.
\label{prop:G}
\end{proposition}

\begin{proof} As $\Omega$ is of class $C^0$,  for each point $\bar P\in\partial\Omega$ there
is a good direction $n_{\bar P}$ at $\bar P$, with corresponding $\delta=\delta(\bar P)$. Then $n_{\bar P}$ is a good direction at any $P\in B(\bar P,\frac{1}{2}\delta(\bar P))$ since for such $P$  we have $\bar P\in B(P,\frac{1}{2}\delta(\bar P))\subset B(\bar P,\delta(\bar P))$.  As $\partial\Omega$ is compact, there exist
$P_i,i=1,\ldots,k,$ such that $\partial\Omega\subset U\stackrel{\rm def}{=}\cup_{i=1}^k
B(P_i,\frac{1}{4}\delta(P_i))$. Consider a partition of unity subordinate  to the covering $\{B(P_i,\frac{1}{2}\delta(P_i))\},i=1,\ldots,k,$ of $\bar U$, namely  functions $\alpha_i\in
C^\infty(\R^m,\R_+),i=1,2,\dots,k$ with $\textrm{supp}
\,\alpha_i\subset B(P_i,\frac{1}{2}\delta(P_i))$ and $\Sigma_{i=1}^k \alpha_i=1$ in $\bar U$. If $P\in U$ and $i\in S_P\stackrel{\rm def}{=}\{j\in\{1,2,\dots,k\}:P\in B(P_j,\frac{1}{2}\delta(P_j))\}$ then $n_{P_i}$ is a good direction at $P$ with respect to the ball $B(P,\Delta(P))$ where $\Delta(P)\stackrel{\rm def}{=}\frac{1}{2}\min_{i\in S_p}\delta(P_i)$.
It then  follows from Lemma \ref{lemma:multiple} that
\be
G(P)\stackrel{\rm def}{=}\frac{\Sigma_{i=1}^k \alpha_i(P)n_{P_i}}{|\Sigma_{i=1}^k \alpha_i(P)n_{P_i}|},\;\mbox{ for all }
P\in U \label{def:G}
\ee  has the required property.   \end{proof}

\begin{definition} We call a field of good directions, constructed by {\rm(\ref{def:G})}, a {\rm canonical field of good directions}.
\end{definition}

\section{ A proper generalized distance}
\label{pgd}

\par For a bounded open set $\Omega$ define the signed distance  $d(x)$ to the boundary
$\partial\Omega$ by
\be
d(x)=\left\{\begin{array}{ll}
\inf_{y\in\partial\Omega}|x-y|  &\textrm{ if $x\in\Omega$}\\
-\inf_{y\in\partial\Omega}|x-y| &\textrm{ if $x\not\in\Omega$}.
\end{array}\right.
\label{def:signed}
\ee

\smallskip
\begin{proposition} Let $\Omega\subset\R^m$ be a bounded domain of class $C^0$.
There exists a function $\rho\in
C^\infty({\R}^m\setminus\partial\Omega)\cap C^{0,1}(\R^m)$ such
that
\be
\frac{1}{2}\le \frac{\rho(x)}{d(x)}\le 2,\,\mbox{ for all }
x\in\R^m\setminus\partial\Omega \label{equivdist}
\ee and
\be
|\nabla\rho(x)|\not=0 \mbox{ for all } x\textrm{ in a neighbourhood of
}\partial\Omega,\quad x\not\in\partial\Omega. \label{nonzerograd}
\ee
\label{lemma:nonzero}
\end{proposition}
\begin{proof} We let $\rho$  be a regularized
distance function, as constructed by Lieberman
\cite{lieberman} (following related earlier work of  Fraenkel \cite{fraenkel}). To define it let $\varphi\in
C^\infty(\R^m)$ be a nonnegative function, whose support is
the unit ball and is such that $\int_{\R^m}\varphi(x)\,dx=1$.
For $x\in \R^m, \tau\in\R$, let
\be
G(x,\tau)\stackrel{\rm def}{=}\int_{|z|<1} d\left(x-{\small\frac{\tau}{2}} z\right)\varphi(z)\,dz.
\label{G}
\ee
\par Since $d$ is $1$-Lipschitz, $|\partial G/ \partial\tau|\leq 1/2$, and so there is a unique solution  $\rho(x)$ of the
equation $\rho(x)=G(x,\rho(x))$. Thus defined, $\rho$ is a Lipschitz function, smooth outside $\partial\Omega$,  that
satisfies (\ref{equivdist}) but not necessarily (\ref{nonzerograd})
(see \cite{lieberman}, Lemma $1.1$ and the comments following it).
\par We continue by proving (\ref{nonzerograd}) for $x$ in a neighbourhood of the boundary and
$x\in\Omega$. To this end we consider a point $P\in \partial\Omega$. Without loss of generality we can suppose that
$P=0$ and that in a suitable
coordinate system  there exist $\delta>0$ and a continuous $f:\R^{m-1}\rightarrow \R$ such that
\be
\mcU_\delta\stackrel{\rm def}{=}\Omega\cap B(0,\delta)=\{y\in\R^m:\,y_m>f(y'),\,|y|<\delta\}.
\label{u1}
\ee
Denoting by $e_m$ the unit vector in the $y_m$ direction, let $y,y+he_m\in {\mathcal U}_{\delta/4}$ for some $h>0$. Then $h<\delta/2$. Since $0\in\partial\Omega$, $d(y)\leq|y|$. If $v\in\overline{B(y,d(y))}$ then
$$|v|\leq|v-y|+|y|\leq d(y)+|y|\leq 2|y|\leq\delta/2.$$
Hence $\overline{B(y,d(y))}\subset U_\delta$. We claim that
\be\label{strict}
d(y+he_m)>d(y).
\ee
If not, there would exist  $w\in\partial \Omega$ with $|w-y-he_m|\leq d(y)$. Thus $w-he_m\in\overline{B(y,d(y))}$ and so $w-he_m\in U_\delta$ and $w_m-h\geq f(w')=w_m$, a contradiction. Since $d$ is Lipschitz it follows that the derivative $\frac{\partial d}{\partial x_m}$ exists a.e. with strictly positive integral on every line segment in ${\mathcal U}_{\delta/4}$ parallel to $e_m$.  By the definition of weak derivatives
\begin{eqnarray}\frac{\partial G}{\partial x_m}(x,\tau)&=&\frac{\partial}{\partial x_m}\int_{\R^m}d\left(x-\frac{\tau}{2}z\right)\varphi(z)\,dz\nonumber\\
&=&\frac{\partial}{\partial x_m}\int_{\R^m}d(\zeta)\varphi\left(\frac{2}{\tau}(x-\zeta)\right)\left(\frac{2}{\tau}\right)^md\zeta\\
&=&\int_{\R^m} d\left(x-\frac{\tau}{2}z\right)\frac{2}{\tau}\varphi_{,m}(z)\,dz\nonumber\\
&=&\int_{\{ |z|<1\}}\frac{\partial d}{\partial x_m}\left(x-\frac{\tau}{2}z\right)\varphi(z)\,dz.
\end{eqnarray}

Suppose now that $x\in {\mathcal U}_{\delta/8}$. Then for $|z|<1$ we have $x-\frac{\rho(x)}{2}z\in \Omega$ and $|x-\frac{\rho(x)}{2}z|\leq \frac{\delta}{8}+d(x)<\frac{\delta}{4}$, where we have used \eqref{equivdist}. Hence, since the partial derivatives equal the weak derivatives almost everywhere in $\Omega$, by Fubini's theorem $\frac{\partial G}{\partial x_m}(x,\rho(x))>0$, and so differentiating $\rho(x)=G(x,\rho(x))$ we obtain
\be
\frac{\partial\rho}{\partial x_m}(x)=\frac{\frac{\partial G}{\partial x_m}}{1-\frac{\partial G}{\partial\tau}}>0.
\label{rhomon}
\ee
Thus every point $P\in \partial \Omega$ has a neighbourhood ${\mathcal U}(P)$ such that $|\nabla \rho(x)|\neq 0$ for $x\in {\mathcal U}(P)\cap\Omega$. By compactness this implies that there is a neighbourhood $\mathcal U$ of $\partial\Omega$ such that $|\nabla \rho(x)|\neq 0$ for $x\in {\mathcal U}\cap\Omega$.
\par The case when $x$ is in a neighbourhood of the boundary
$\partial\Omega$ but $x\in \R^m\setminus \overline\Omega$ is treated
similarly.  \end{proof}

\smallskip
\begin{remark}\label{remark:nonnegative angle}
{\rm For $\Omega\subset\R^m$ bounded, of class $C^0$, the compactness of $\partial\Omega$ implies that there exists a neighbourhood $\mathcal U$  of $\partial\Omega$ such that  $\mathcal U\subset\cup_{i=1}^k B(P_i,\delta_i)$ where $k\geq 1$ and, for all $i=1,\dots,k$,
$P_i\in\partial\Omega$ and at $P_i$ there is a good direction $n_i\in\Sph^{m-1}$ with respect to the ball $B(P_i,8\delta_i)$. Then relation \eqref{rhomon} in the previous proof shows that for any $P\in\mcU\setminus\partial\Omega$ we have   
$$
\frac{\partial{\rho}}{\partial n_j}(P)=n_j\cdot\nabla \rho(P)>0
$$ for those $j\in \{1,\dots, k\}$ such that  $P\in B(P_j,\delta_j)$. 

 Moreover, for any $n$ that is a convex combination of those good directions $n_j$, $j\in \{1,\dots, k\}$ with  $P\in B(P_j,\delta_j)$, we have
$$
\frac{\partial{\rho}}{\partial n}(P)>0.
$$ 
} 
\end{remark}

\begin{remark}\label{remark:boundarycaserho}
{\rm 
 We claim now that for any $R\in\mcU$ (with $\mcU$ as in Remark \ref{remark:nonnegative angle}) and any $n$ that is a convex combination of those good directions $n_j$,  $j\in \{1,\dots, k\}$, such that  $R\in B(P_j,\delta_j)$, we have that $n$ is also a good direction at $R$ and there exists $\delta_n>0$ such that  
$$
\rho(R+sn)<\rho(R+tn)
$$ for all $s,t\in (-\delta_n,\delta_n)$ with $s<t$.

If $R\not\in\partial\Omega$ then Remark \ref{remark:nonnegative angle} suffices for obtaining the claim. If $R\in\partial\Omega$  we consider the function 
$h:[-1,1]\to \R$ defined by $h(\tau)=\rho(R+\tau n)$. Then Remark \ref{remark:nonnegative angle} ensures that $h'(\tau)>0$ for $\tau\in (-\delta_n,0)\cup(0,\delta_n)$ for some $\delta_n>0$. This fact, together with $h(0)=0$ and  $h(\tau)\tau>0$ for $\tau\in (-\delta_n,0)\cup(0,\delta_n)$ (since  $n$ is a pseudonormal at $R$) suffices to obtain the claim in this case as well.
} 
\end{remark}

\section{The flow of   canonical good directions}
\label{section:flow}

\medskip\par  We continue working with $\Omega\subset\R^m$ a bounded domain of class $C^0$ and $\rho$ a proper regularized distance from the boundary $\partial\Omega$ as described in Section \ref{pgd}. We take  a function $\gamma\in C^\infty(\R^m, \R_+)$ so that  ${\rm supp}\,\gamma\subset  U$ (where  $U$ is as in Proposition~\ref{prop:G}  and $U\subset\mcU$ with $\mcU$  as in Remark~\ref{remark:nonnegative angle}) with $\gamma\equiv 1$ on $\overline{W}$
 where $W=\{x\in \R^m: |\rho(x)|<\bar\eps\}$ and $\bar\eps>0$ is small enough so that $W\subset U$, and $\gamma\leq 1$  on $\R^m\setminus \overline{W}$.
 Let $G:U\to\Sph^{m-1}$ be the function from Proposition~\ref{prop:G}, so that $G(P)$ is a good direction at $P$. Let $S(t)x_0$ denote the solution at time $t\in\R$ of the system:
\be
          \dot x(t)=\left\{\begin{array}{ll}\gamma(x(t))G(x(t)) \, &\textrm{ for $t\in\R$, $x(t)\in U$}\\ 0 &\textrm{ for $t\in\R$, $x(t)\not\in U$}\end{array}\right.  
\label{-G}
\ee
with initial data $x(0)=x_0$.
\par  From now on we call the globally defined flow $S(\cdot)(\cdot):\R\times \R^m\to\R^m$ {\it the flow of canonical good directions}.
 \par   We first  show that the regularized distance to the boundary increases along this flow.
\begin{lemma}\label{lemma:distdecay}Let $\Omega\subset\R^m$ be a bounded domain of class $C^0$. 
 Let $P\in U$, with $\gamma(P)\not=0$ (where $U$ and $\gamma$ are defined at the beginning of the section). Then
\be
\rho\left(S(\mu_1)P\right)<\rho\left(S(\mu_2)P\right),\mbox{ for all } 0\leq \mu_1<\mu_2.
\ee
\end{lemma}
\begin{proof}We proceed in two steps:
\smallskip\par{\it Step 1.} We claim that for any point $P\in U$ with $\gamma(P)\neq 0$  we have   $\rho\left(S(\mu_1)P\right)\le \rho\left(S(\mu_2)P\right)$
for $0\le\mu_1<\mu_2$. Then $\gamma(S(t)P)\neq 0$ for all $t\geq 0$. 
\par We consider the Euler polygonal approximation of the system (\ref{-G}),  on the interval $[0,\mu_2+1]$. This is obtained
by linearly interpolating between the points $P_k$ defined  recursively by:
$$P_{k+1}=P_k+h\gamma(P_k)G(P_k),\;\; h=\frac{\mu_2+1}{l},\;\; k=0,\dots, l-1, $$ where $P_0\defeq P$.
\par Thus we have the approximate solution $S_l(t)=P_k+(t-k\frac{\mu_2+1}{l})\gamma(P_k)G(P_k)$ for all $t\in [k\frac{\mu_2+1}{l},(k+1)\frac{\mu_2+1}{l}],k=0,\dots, l-1$. Note that by the convergence of the Euler approximation, for $h$ small enough $\gamma(P_k)\neq 0$ for all $k=0,\dots,l-1$.  Using then  Remark~\ref{remark:boundarycaserho} we have that for $h$ small enough $\rho$ is an increasing function along $S_l(t),t\in [0,\mu_2+1]$. Passing to the limit $l\to \infty$ we have that $\rho$ is a non-decreasing function along the limit function $S(t),t\in [0,\mu_2+1]$ that is also a solution of the system (\ref{-G}).
\smallskip\par{\it Step 2.}  We claim now that for any $\mu_1<\mu_2$ we have $\rho(S(\mu_1)P)<\rho(S(\mu_2)P)$.  To this end we claim first that for any $\eps<(0,\mu_2-\mu_1)$ there exists an interval $(a,b)\subset (\mu_1,\mu_1+\eps)$   and a subset $B\subset \{1,2,\dots, k\}$ so that
 $\alpha_i(S(t)P)\not=0$ for all $t\in (a,b),i\in B$ and $\alpha_i(S(t)P)=0$ for all $t\in (a,b),i\in \{1,2,\dots,k\}\setminus B$ (where the functions $\alpha_i$ are those used in the definition of $G$ in the proof of Proposition~\ref{prop:G}).

 In order to prove the claim let $M_i\defeq\{ t\in [\mu_1,\mu_1+\eps]: \alpha_i(S(t)P)>0\}$. Then each $M_i, 1\le i\le k,$ is relatively open in $[\mu_1,\mu_1+\eps]$ and the $M_i$ cover $[\mu_1,\mu_1+\eps]$. Each $\partial M_i$ is closed and nowhere dense. Hence by the Baire category theorem (for a finite number of sets) $\cup_{i=1}^k \partial M_i$ is closed and nowhere dense. Let $(a,b)\subset \left(\cup_{i=1}^k \partial M_i\right)^c$. Then $B(t)\defeq\{ i\in \{1,\dots,k\}:t\in M_i\}$ is constant in $(a,b)$ and we can take $B=B(t)$, thus finishing the proof of the claim.

\par We consider now the function $S(t)P$ for $t\in (a,b)$ with $(a,b)\subset(\mu_1,\mu_1+\eps)$ as in the claim above. Let us denote $R\defeq S(a)P$ and let $B=\{i_1,i_2,\dots,i_j\}$. Then $S(a+s)P=S(s)R$ and we have (for $s<b-a$) that
\begin{align} 
&\hspace{-.7in}S(a+s)P=R+  \sum_{r=1}^j n_{i_r}\int_0^s\frac{\gamma(S(\tau)R)\alpha_{i_r}(S(\tau)R)}{|\sum_{r=1}^jn_{i_r}\alpha_{i_r}(S(\tau)R)|}d\tau\nonumber\\
&=R+\sum_{r=1}^j n_{i_r} \xi_{i_r}=R+\left(\sum_{r=1}^j \xi_{i_r}\right)
\left(\sum_{p=1}^j\frac{\xi_{i_p}}{\left(\sum_{r=1}^j \xi_{i_r}\right)}n_{i_p}\right)
\label{trick_good}
\end{align} 
where we denote $\xi_{i_r}\stackrel{\rm def}{=}\int_0^s\frac{\gamma(S(\tau)R)\alpha_{i_r}(S(\tau)R)}{|\sum_{r=1}^jn_{i_r}\alpha_{i_r}(S(\tau)R)|}d\tau>0,\,r=1,\dots,j$. By our choice of the set of indices $B$, we have that $n_{i_p}$ is a good direction at $R$ for all $i_p\subset B,\,p=1,\dots,j,$ and thus their convex combination $\left(\sum_{p=1}^j\frac{\xi_{i_p}}{\left(\sum_{r=1}^j \xi_{i_r}\right)}n_{i_p}\right)$ is also a good direction at $R$.  Using then  Remark~\ref{remark:boundarycaserho} we have that $\rho(S(a)P)<\rho(S(a+s)P)$ for $\eps>0$ small enough and arbitrary  $s\in (0,\eps)$, which combined with Step $1$ completes the proof.  \end{proof}
\bigskip\par We now show that in a neighbourhood of the boundary the flow of canonical good directions crosses the boundary uniformly in time.
\begin{lemma} Let $\Omega\subset\R^m$ be a bounded domain of class $C^0$. If   $\bar\varepsilon$ is as defined at the beginning of the section and  $0\leq\varepsilon\leq\bar\varepsilon$  then

{\rm(i)} there exists $t^\varepsilon_-<0$ such that 
\begin{equation}\rho(S(t)P)<-\varepsilon,\textrm{ for all }t\leq t^\varepsilon_-, P\in \overline W,\label{flow1}\end{equation}

{\rm(ii)} there exists $t^\varepsilon_+>0$ such that  
\begin{equation}\rho(S(t)P)>\varepsilon,\textrm{ for all }t\geq t^\varepsilon_+, P\in \overline W,\label{flow2}\end{equation}
 where $W\subset U$ is as defined at the beginning of the section.
\label{lemma:limitgood}
\end{lemma}

\begin{proof} We consider case (i), the argument for case (ii) being similar. Let us denote by $Z\subset U$ the
$\alpha$-limit set of the solution $S(t)P$. Since $S(\R)P$ is bounded (because all points outside a neighbourhood of the boundary are stationary points), $Z$ is a nonempty, compact, invariant 
set that attracts $P$ along the flow $S$ (see, for instance,   
\cite[Lemma 3.1.1]{hale}). 
\par We show first that the conclusion is true if we let $t^\varepsilon_-$ depend on $P$.
 We argue by contradiction and assume that the conclusion is
false, so that there exists a sequence  $t_k\to -\infty$ such that
$\rho(S(t_k)P)\ge -\varepsilon$. As $\rho(S(t)P)$ is increasing and bounded from
below we have that $\lim_{t\to-\infty} \rho(S(t)P)=l\ge -\varepsilon$ and $\rho(Q)=l$
for all $Q\in Z$. In particular $Z\subset \bar W$. As $Z$ is invariant, 
  for each $Q\in Z$ we also have $S(t)Q\in Z$ for all  $t\le 0$
 and $\rho(S(t)Q)=l$ for all $t\le 0$, which contradicts  Lemma~\ref{lemma:distdecay}. In order to prove that $t^\varepsilon_-$
can be chosen independent of $P$ we assume for contradiction that
this is not possible, so that there exist a sequence
$\{P_k\}_{k\in\mathbb{N}}\subset\overline W$ and a
corresponding sequence of times $\{t_k\}_{k\in\mathbb{N}}$ so that
 $\rho(S(t_k)P_k)=-\varepsilon$ and $t_k\to -\infty$. Using the compactness of
$\overline W$ we can find a subsequence $P_{k_l}\to
P_0\in\overline W$. But for $P_0$ there exists a time $t_0<0$
such that $\rho(S(t_0)P)<-\varepsilon$ and using the continuity with respect to
the initial data for the solution of the system (\ref{-G}) together with the fact that $\rho(S(t)P_k)$ is increasing, we obtain
a contradiction. \end{proof}

\section{Homeomorphically and $C^\infty$ diffeomorphically equivalent approximations of rough domains}
\label{section:diffeomorphism}

\par In this section we provide an application of the tools developed in the previous sections by showing that one can approximate
from the inside (and  also from the outside) domains $\Omega$ of class $C^0$ by smooth domains $\Omega'$ such that $\Omega$ and $\Omega'$ are $C^\infty$-diffeomorphic and their closures are homeomorphic.

\begin{theorem}
\label{homapprox}Let $\Omega\subset\R^m, m\geq 2$ be a bounded domain of
class $C^0$. Let $\rho$ be a regularized distance as given in Proposition~{\rm \ref{lemma:nonzero}} and for $\varepsilon\in\R$ define
\begin{equation}
\label{omegaep}
\Omega_\varepsilon=\{x\in\R^m:\rho(x)>\varepsilon\}.
\end{equation}
 There exists $\varepsilon_0=\varepsilon_0(\Omega)>0$ such that if $0<|\varepsilon|<\varepsilon_0$ then $\Omega_\varepsilon$ is a bounded domain of class
 $C^\infty$ and
 
{\rm (i)} $\bigcap_{-\varepsilon_0<\varepsilon<0}\Omega_\varepsilon=\overline\Omega,\; \bigcup_{\varepsilon_0>\varepsilon>0}\Omega_\varepsilon=\Omega$,\; $\bar\Omega_\eps\subset \Omega_{\eps'}\mbox{ if }-\eps_0<\eps'<\eps<\eps_0$.

{\rm(ii)}  For $0\leq|\varepsilon|<\varepsilon_0$ there exists a  homeomorphism $f(\varepsilon,\cdot)$ of $\R^m$ onto $\R^m$, with inverse denoted $f^{-1}(\varepsilon,\cdot)$, such that   
\begin{itemize}
  \item  $f(\varepsilon,\bar\Omega)=\overline\Omega_\varepsilon,\;f(\varepsilon,\partial\Omega)=\partial\Omega_\varepsilon$,
  \item   $f(\varepsilon,x)=x$ for $|\rho(x)|>3|\varepsilon|$ $($so that in particular $f(0,\cdot)=\,$identity$)$,
\item  $f(\varepsilon,\cdot):\R^m\setminus\partial\Omega\rightarrow \R^m\setminus\partial\Omega_\varepsilon$ is a   $C^\infty$ diffeomorphism.
\end{itemize}
Furthermore $f$ and $f^{-1}$ are  continuous functions of $(\varepsilon,x)$ for $0\leq|\varepsilon|<\varepsilon_0$, $x\in\R^m$, and $f$ $($resp. $f^{-1})$ is   a smooth function
of $(\varepsilon,x)$ for $0<|\varepsilon|<\varepsilon_0$, $x\in\R^m\setminus\partial\Omega$ $($resp. $0<|\varepsilon|<\varepsilon_0$, $x\in\R^m\setminus\partial\Omega_\varepsilon)$.

{\rm (iii)} 
 There exists a map $\bar f:(0,\varepsilon_0)\times (-\varepsilon_0,0)\times \R^m\to\R^m$ such that if $0<\varepsilon<\varepsilon_0, -\varepsilon_0<\varepsilon'<0$ then 
\begin{itemize}
  \item $\bar f(\varepsilon,\varepsilon', \cdot)$ is 
 a $C^\infty$ diffeomorphism  of $\R^m$ onto $\R^m$ with inverse $\bar f^{-1}(\varepsilon,\varepsilon', \cdot):\R^m\to\R^m$,  
\item $\bar f(\varepsilon, \varepsilon', \Omega_\varepsilon)= \Omega_{\varepsilon'}$, $\bar f(\varepsilon,\varepsilon', \partial\Omega_\varepsilon)=\bar f(\varepsilon, \varepsilon', \partial \Omega_{\varepsilon'})$,
\item   $\bar f(\varepsilon,\varepsilon',x)=x$ if $\rho(x)<3\varepsilon'$ or $\rho(x)>3\varepsilon$.
  \end{itemize}\item Furthermore $\bar f$ and $\bar f^{-1}$ are smooth functions of $(\varepsilon,\varepsilon', x)\in(0,\varepsilon_0)\times(-\varepsilon_0,0)\times\R^m$.
\end{theorem}

\begin{proof} We choose $\varepsilon_0>0$ small enough so that $3\varepsilon_0<\bar\varepsilon$, so that we can use in $\Omega_{3\varepsilon_0}$ all the constructions from the previous section. Conclusion (i) is then immediate. 

In order to prove (ii)   we begin by considering the problem of approximating $\Omega$ from the interior, and  construct the desired  function $f$ first just on $[0,\varepsilon_0)\times \overline\Omega$. Thus we assign to each $x\in \overline\Omega$ and $\varepsilon\in [0,\varepsilon_0)$ a value $f(\varepsilon,x)\in\overline\Omega_\varepsilon$ taken to be along the flow $S(\cdot)x$ defined in \eqref{-G}, starting at $x$.  However, since the flow  is defined to be non-stationary just in a neighbourhood of the boundary, we  take $f(\varepsilon,x)=x$ for $x$ far enough from the boundary. Thus we define
\begin{equation}
\label{jb1}
f(\varepsilon,x)=\left\{\begin{array}{ll} S(t(\varepsilon,x))x, & x\in \overline\Omega\setminus\Omega_{3\varepsilon}, \\
x,& x\in\Omega_{3\varepsilon},
\end{array}\right.
\end{equation}  for a $t(\varepsilon,x)$ to be determined, where $t(0,x)=0$ (so that $f(0,x)=x$). In order to define $t(\varepsilon,x)$ for $\varepsilon\in (0,\varepsilon_0)$ let $\tilde h:\R_+\to [0,1]$   be a smooth function such that $\tilde h\equiv 1$ on $[0,1]$,  $\tilde h\equiv 0$ on $[5/2,\infty)$ and $-1<h'(r)\leq 0$ for all $r\geq 0$. We take now $h(\varepsilon,r)\defeq \varepsilon\tilde h(\frac{r}{\varepsilon})$. Then 
$h(\varepsilon,\cdot)\equiv \varepsilon$ on $[0,\varepsilon]$ and $h(\varepsilon,\cdot)\equiv 0$ on $[5\varepsilon/2,\infty)$ with $-1<\frac{\partial h}{\partial r}(\varepsilon,r)\le 0$ for all $r\geq 0,\,\varepsilon\in (0,\varepsilon_0)$. For $x\in \overline\Omega\setminus\Omega_{3\varepsilon}$  and $\varepsilon\in (0,\varepsilon_0)$  define  $t(\varepsilon, x)$ to be the unique $t\ge 0$ such that                           
\begin{equation}\label{eq:tx}
\rho(S(t)x)=\rho(x)+h(\varepsilon,\rho(x)).                           
\end{equation}
\begin{figure}[h] \centering \def\svgwidth{300pt} 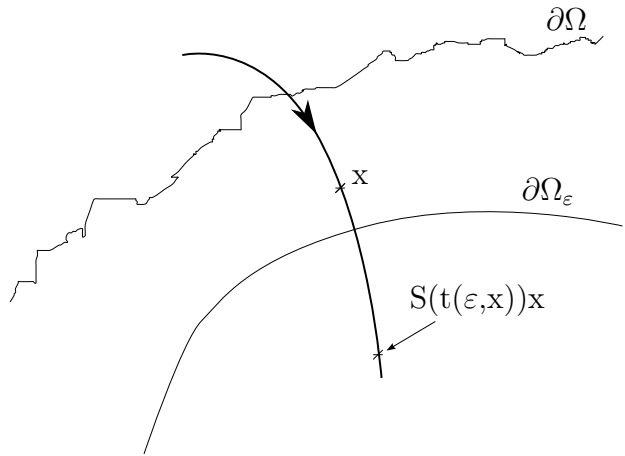
\caption{Defining the diffeomorphism along the flow}
\end{figure}
We claim now that  $t(\varepsilon,x)$  is well defined. We denote $g(t)\stackrel{\rm def}{=}\rho(S(t)x)-\rho(x)-h(\eps,\rho(x))$. Then $g(0)\le 0$. But $\rho(S(\bar t)x)=3\varepsilon$ for some $\bar t\geq 0$ because of Lemma~\ref{lemma:limitgood} and the intermediate value theorem. Thus $g(\bar t)=3\varepsilon-\rho(x)-h(\eps,\rho(x))\geq 0$, since $\rho(S(\bar t)x)\geq\rho(x)$ and   $h(\varepsilon,3\varepsilon)=0$ together with  $\frac{\partial}{\partial\rho}(\rho+h(\eps,\rho))>0$. Finally  $g(t)$ is strictly increasing by Lemma \ref{lemma:distdecay}. This proves our claim regarding the definition of $t(\varepsilon,x)$. 
Note that the properties of $t$ imply that  $f(\varepsilon,\cdot):\partial\Omega\rightarrow\partial\Omega_\varepsilon$ and $f(\varepsilon,\cdot):\Omega\rightarrow\Omega_\varepsilon$.

\par We continue by claiming that $t$ is a smooth function   of   $(\varepsilon,x)$  in $(0,\varepsilon_0)\times \left(\Omega\setminus\bar\Omega_{3\varepsilon}\right)$. This follows from the implicit function theorem applied to 
\begin{equation}\label{eq:Fimplicit}
F(t,\varepsilon, x)\stackrel{\rm def}{=}\rho(S(t)x)-\rho(x)-  h(\varepsilon,\rho(x)) 
\end{equation} (the non-degeneracy condition needed for applying the implicit function theorem is a consequence of the relation
\begin{equation}\label{implicit:non-degeneracy}
\frac{d}{dt}\rho(S(t)x)=\nabla\rho(S(t)x)\cdot \frac{d}{dt} S(t)x=\left[(\nabla\rho\cdot G)\gamma\right](S(t)x)>0,
\end{equation}
where    for the equality we used the definition of the flow $S(\cdot)x$ and for the   last inequality we used   the definition of $G$ together with Remark~\ref{remark:nonnegative angle} with $x=P$).

Since $S(t)x,\,\rho(x),\,h(\varepsilon,\rho(x))$ are smooth  in   $t$, $x\in\Omega$  and $\varepsilon>0$ we deduce  that  $f$ is smooth on $(0,\varepsilon_0)\times\left(\Omega\setminus \overline\Omega_{3\varepsilon}\right)$. Since if $\rho(x)\in [5\varepsilon/2, 3\varepsilon]$ then $t(\varepsilon,x)=0$, and thus $f(\varepsilon,x)=x$,  it follows that $f$ is smooth in $(0,\varepsilon_0)\times \Omega$. 
 To show that $f:[0,\varepsilon_0)\times\overline\Omega\rightarrow\overline\Omega_\varepsilon$ is continuous it is enough to show that   $t(\cdot,\cdot)$ is continuous at points $(\eps,\tilde x)$ where $\eps\in[0,\eps_0)$ and $\tilde x\in\partial\Omega$.    Assume for contradiction that this is not the case, so that there exist   sequences $x_k\to \tilde x$, $x_k\in \overline\Omega$ and $\eps_k\to\eps$ in $[0,\eps_0)$ such that $t(\varepsilon_k,x_k)$ does not converge to $t(\varepsilon,\tilde x)$. By the uniformity in time in Lemma~\ref{lemma:limitgood}  we may assume that $t(\varepsilon_k,x_k)\to \tau\not=t(\varepsilon,\tilde x)$. We may also suppose without loss of generality that $\eps_k>0$ for all $k$.  Replacing $x$ with $x_k$ in \eqref{eq:tx} and passing to the limit $k\to\infty$ we obtain that $\rho\left(S(\tau)\tilde x\right)=\rho\left(S(t(\varepsilon,\tilde x))\tilde x\right)$, and  hence $\tau=t(\varepsilon,\tilde x)$, a contradiction which proves our assertion that $t$ is continuous up to the boundary.  
 
Next we check that $f(\varepsilon,\cdot)$ is one-to-one. Suppose $f(\varepsilon,x)=f(\varepsilon,y)$ for $x,y\in\overline\Omega$. 
 If $\rho(x)> 3\varepsilon$ and $\rho(y)> 3\varepsilon$ then $f(\varepsilon,x)=x,\,f(\varepsilon,y)=y$ and so $x=y$.
  If $\rho(x)\leq 3\varepsilon$ and $\rho(y)> 3\varepsilon$ then $\rho(f(\varepsilon,x))=\rho(x)+h(\varepsilon,\rho(x))\leq 3\varepsilon +h(\varepsilon,3\varepsilon)=3\varepsilon<\rho(y)=\rho(f(\varepsilon,y))$ so this case cannot occur.
 Finally if both $\rho(x)$ and $\rho(y)$ are in $[0,3\varepsilon]$ then we have $S(t(\varepsilon,x))x=S(t(\varepsilon,y))y$ and hence $\rho(x)=\rho(y)$. If $t(\varepsilon,x)=t(\varepsilon,y)$ there is nothing to prove, so we assume without loss of generality that $t(\varepsilon,x)>t(\varepsilon,y)$. Then $S(t(\varepsilon,x)-t(\varepsilon,y))x=y$, hence $\rho(x)<\rho(y)$ since $\rho(S(t)x)$ is strictly increasing in $t$, giving a contradiction.

We next show that  $f(\varepsilon,\cdot)$ is onto. This is obvious if $\eps=0$ so we suppose $\eps\in(0,\eps_0)$. To this end we take  an arbitrary $z\in\bar\Omega_\eps$ and seek $x\in\bar\Omega$ with $f(\eps,x)=z$. If $\rho(z)\ge 3\varepsilon$ then $f(\varepsilon,z)=z$, so we suppose $\rho(z)<3\varepsilon$. First note that by Lemma \ref{lemma:limitgood} there exists $\alpha(\varepsilon, z)\leq 0$ with $\rho(S(\alpha(\varepsilon, z))z)=0$. We look for $x$ of the form $x=S(\beta(\varepsilon, z))z$ with $\alpha(\varepsilon, z)\leq\beta(\varepsilon, z)\le 0$. Denoting
\begin{equation}
\label{jb2}
\bar g(\varepsilon,z,\tau)\defeq\rho(z)-\rho(S(\tau)z)-h(\varepsilon,\rho(S(\tau)z)),
\end{equation}
we have that  $\bar g(\varepsilon,z,0)\leq 0, \bar g(\varepsilon,z,\alpha(\varepsilon, z))=\rho(z)-\varepsilon\geq 0$. Since  $\bar g(\varepsilon,z,\tau)$ is strictly decreasing in $\tau$ it follows that there exists $\beta(\varepsilon,z)\in[\alpha(\varepsilon,z),0]$ with $\bar g(\varepsilon,z,\beta(\varepsilon,z))=0$, that is
$$\rho(S(-\beta(\varepsilon,z))S(\beta(\varepsilon, z))z)=\rho(S(\beta(\varepsilon, z))z)+h(\varepsilon,\rho(S(\beta(\varepsilon, z))z)).$$
Also, since $\beta(\varepsilon, z)\geq \alpha(\varepsilon, z)$ we have that $\rho(S(\beta(\varepsilon,z))z)\geq \rho(S(\alpha(\varepsilon, z))z)=0$,  so that $S(\beta(\varepsilon, z))z\in\overline\Omega$. Hence $t(\varepsilon,S(\beta(\varepsilon, z))z)=-\beta(\eps, z)$ and $f(\varepsilon,S(\beta(\varepsilon, z))z)=S(-\beta(\varepsilon, z))S(\beta(\varepsilon, z))z=z$. 
Hence $\bar f(\eps,\cdot):\bar\Omega\to\bar\Omega_\eps$ is onto, and so also $\bar f(\eps,\cdot): \Omega\to \Omega_\eps$ is onto. Furthermore, relation   \eqref{implicit:non-degeneracy} implies that $ \frac{\partial \bar g}{\partial \tau}(\varepsilon,z,\beta(\varepsilon, z))<0$, so that by the implicit function theorem  $\beta$ is a smooth function of $(\eps,z)$  for $z\in\Omega_\eps$ and $\varepsilon\in (0,\varepsilon_0)$. Therefore $f^{-1}(\varepsilon, z)=S(\beta(\varepsilon,z))z$ is also a smooth function of $(\varepsilon,z)\in (0,\varepsilon_0)\times \Omega_\varepsilon$. That $f^{-1}:[0,\eps_0)\times\bar\om_\eps\to\bar\om$ is continuous follows from the continuity of $f$. 

This completes the construction of the mapping $f=f(\eps,x)$ in (ii) for $\eps\in[0,\eps_0)$ and $x\in\bar\om$.  In particular, since $\Omega$ is by hypothesis connected, so is $\Omega_\varepsilon$ for $\eps\in[0,\eps_0)$.  To show that $\Omega_\varepsilon$ is of class $C^\infty$ for $\eps\in(0,\eps_0)$ let $\bar x\in\partial\Omega_\varepsilon$, so that $\rho(\bar x)=\varepsilon$. As $|\nabla\rho(\bar x)|\neq 0$ at least one of the partial derivatives $\frac{\partial\rho}{\partial x_i}(\bar x)$ is nonzero. Without loss of generality we may assume that $\frac{\partial\rho}{\partial x_m}(\bar x)>0$. By the implicit function theorem there exist $\delta>0$ and a function $\tilde f\in C^\infty(\R^{m-1})$ such that
$$\{x\in B(\bar x,\delta):\rho(x)=\varepsilon\}=\{(x',x_m)\in B(\bar x,\delta):x_m=\tilde f(x')\}.$$
Since for $\delta>0$ sufficiently small $\frac{\partial\rho}{\partial x_m}(x)>0$ for all $x\in B(\bar x,\delta)$ it follows that $\Omega_\varepsilon\cap B(\bar x,\delta)=\{(x',x_m)\in\Omega_\varepsilon:x_m>\tilde f(x')\}$ as required.

 Notice that   the same argument proves that there is a continuous mapping $F:(-\eps_0,0]\times\om^c\to\R^m$ such that if $\eps'\in(-\eps_0,0]$ then $F(\eps',\cdot)$ is a homeomorphism of $\om^c$ onto $\om_{\eps'}^c$ with  inverse $F^{-1}(\eps',\cdot):\om_{\eps'}^c\to\om^c$, such that $F(\eps',\cdot)$ is a $C^\infty$ diffeomorphism of $(\bar\om)^c$ onto $(\bar\om_{\eps'})^c$ if $\eps'\in(-\eps_0,0)$, and such that $F(\eps',x)=x$ if $\rho(x)<3\eps'$. Furthermore $F^{-1}:(-\eps_0,0]\times\om_{\eps'}^c\to\om^c$ is continuous, while $F:(-\eps_0,0)\times (\bar\om)^c\to(\bar\om_{\eps'})^c$ and  $F^{-1}:(-\eps_0,0)\times(\bar\om_{\eps'})^c\to(\bar\om)^c$ are smooth. In fact we can deduce this by applying the above to each of the finite number of  connected components of the bounded open set $\om_R=B(0,R)\setminus\bar\om$ for a large $R$ with $\bar\om\subset B(0,R)$, noting that each such component is a bounded domain of class $C^0$, and redefining the mapping to be the identity in a suitable neighbourhood of $\partial B(0,R)$. We will use this observation below when extending the definition of $f$ to the whole of $(-\eps_0,\eps_0)\times\R^m$.

\medskip
In order to   define $f$ in (ii) on $(-\eps_0,0]\times\bar\om$ we cannot proceed  in exactly the same way as we did to define $f$ on $[0,\eps_0)\times\bar\om$   via an analogue of the definition \eqref{jb1}, because we would then not be able to prove smoothness via the implicit function theorem at points on, or with images on,  $\partial\Omega$ (since $\rho$ is not smooth there). Instead we proceed by first proving  (iii), from which the extension of $f$ to the whole of $(-\eps_0,\eps_0)\times\R^m$ will follow easily.

We make use of the following lemma.
\begin{lemma}
\label{smoothfunction}
There exists a smooth function $h:(0,1)\times(0,1)\times\R\to\R$ satisfying:
\begin{eqnarray*}
h(a,c,t)=t&&\mbox{for all } a,c\in(0,1), t\not\in[0,1],\\
\frac{\partial h}{\partial t}(a,c,t)>0&&\mbox{for all }a,c\in(0,1), t\in\R,\\
h(a,c,a)=c.&&
\end{eqnarray*}.
\end{lemma}
\begin{proof}
Let $\varphi\in C_0^\infty(\R),\,\varphi\geq 0,\,\varphi(t)=\varphi(-t)$ for all $t$, $\supp\varphi\subset(-1,1)$, $\int_{-\infty}^\infty\varphi(t)\,dt=1$. Define $\delta=\delta(a,c)$ by 
$$\delta(a,c):=\frac{1}{4}a(1-a)c(1-c),$$
and 
$$H(a,c,t):=\left\{\begin{array}{ll}t&\mbox{if }t\not\in[\delta,1-\delta],\\
\delta+(t-\delta)\frac{c-2\delta}{a-2\delta}&\mbox{if }t\in[\delta,a-\delta],\\
t+c-a&\mbox{if }t\in[a-\delta,a+\delta],\\
c+\delta+(t-a-\delta)\frac{1-c-2\delta}{1-a-2\delta}&\mbox{if }t\in[a+\delta,1-\delta].
\end{array}\right.
$$
Note that $\min(a,1-a,c,1-c)>4\delta$, so that the continuous piecewise affine function $H(a,c,\cdot)$ is well defined and has strictly positive slope. We claim that 
\begin{eqnarray*}
h(a,c,t)&:=&\int_{-\infty}^\infty\frac{2}{\delta}\phi\left(\frac{2(t-s)}{\delta}\right)H(a,c,s)\,ds\\&=&\int_{-\infty}^\infty\varphi(\sigma)H(a,c,t+\frac{\delta\sigma}{2})\,d\sigma
\end{eqnarray*}
has the required properties. First, splitting the range of integration into  the different parts in which $H(a,c,\cdot)$ is affine, we see that $h(a,c,t)$ is the sum of five  integrals each having the form
$$\int_{p_i(a,c,t)}^{q_i(a,c,t)}\varphi(\sigma)\theta_i(a,c,t,\sigma)\,d\sigma,$$
where, for each $i=1,\ldots,5$, $p_i,q_i$ are  smooth functions of $(a,c,t)$ in the set $D:=(0,1)^2\times\R$, and $\theta_i$ is smooth on $(0,1)^2\times\R^2$, so that $h$ is smooth on $D$. Since
$$h_t(a,c,t)=\int_{-\infty}^\infty \varphi(\sigma)H_t(a,c,t+\frac{\delta\sigma}{2})\,d\sigma$$
we have that $h_t(a,c,t)>0$. If $t\not\in(\frac{\delta}{2},1-\frac{\delta}{2})$ then $t+\frac{\delta}{2}\not\in(\delta,1-\delta)$ for $|\sigma|<1$, so that, since $\varphi(t)$ is even, $h(a,c,t)=t$. Similarly 
$$h(a,c,a)=\int_{-1}^1\varphi(\sigma)(c+\frac{\delta\sigma}{2})\,d\sigma=c.$$
\end{proof}
Continuing with the proof of (iii), given any $x\in\R^m$ with $|\rho(x)|<\eps_0$ and any $\tau\in(-\eps_0,\eps_0)$, there exists a unique $t=t(\tau,x)$ such that $\rho(S(t)x)=\tau$. Note that if also $|\rho(S(\sigma)x)|<\eps_0$ then
\begin{equation}
\label{timetranslate}
t(\tau, S(\sigma)x)=t(\tau,x)-\sigma.
\end{equation}Furthermore $t(\tau,x)$ is smooth in $\tau\ne 0$ and $x$, and $\frac{d}{d\tau}t(\tau,x)>0$ for $\tau\neq 0$. For such $x$ and $\eps\in(0,\eps_0), \eps'\in(-\eps_0,0)$, define
\begin{equation}\label{alphabeta}
\alpha(\eps,\eps',x)=\frac{1}{t(2\eps,x)-t(2\eps',x)},\;\;\beta(\eps,\eps',x)=-\frac{t(2\eps',x)}{t(2\eps,x)-t(2\eps',x)}.
\end{equation}
Note that by \eqref{timetranslate}, if $y=S(\sigma)x$ and $|\rho(x)|, |\rho(y)|<\eps_0$ then
\begin{equation}
\label{alphabeta1}
\alpha(\eps,\eps',x)=\alpha(\eps,\eps',y),\;\; \beta(\eps,\eps',x)=\beta(\eps,\eps',y) -\sigma\alpha(\eps,\eps',y).
\end{equation}
We define the desired   $\bar f$ in (iii) by 
\begin{equation}\label{jb5}
\bar f(\varepsilon,\varepsilon',x)=\left\{\begin{array}{ll}S(\eta(\varepsilon,\varepsilon',x))x,& \mbox{if }\rho(x)\in (3\eps',3\eps),\\
x,&\mbox{otherwise},\end{array}\right. 
\end{equation}
where 
\begin{equation}\label{sub1}
\eta(\eps,\eps',x)=\alpha(\eps,\eps',x)^{-1}[h(r(\varepsilon,\varepsilon',x),s(\varepsilon,\varepsilon',x),\beta(\varepsilon,\varepsilon',x))-\beta(\varepsilon,\varepsilon',x)],
\end{equation}
$h$ is as in Lemma \ref{smoothfunction} and
\begin{eqnarray}r(\varepsilon,\varepsilon',x)&=&\alpha(\varepsilon,\varepsilon',x)t(\eps,x)+\beta(\varepsilon,\varepsilon',x),\label{sub2}\\
s(\varepsilon,\varepsilon',x)&=&\alpha(\varepsilon,\varepsilon',x)t(\eps',x)+\beta(\varepsilon,\varepsilon',x).\label{sub3}
\end{eqnarray}
Note that  $r(\varepsilon,\varepsilon',x),s(\varepsilon,\varepsilon',x)\in(0,1)$ so that $\bar f(\varepsilon,\varepsilon',x)$ is well defined. Also, from \eqref{alphabeta1}, if $y=S(\sigma)x$ and $|\rho(x)|, |\rho(y)|<\eps_0$ then
\begin{equation}\label{rs}
r(\eps,\eps',x)=r(\eps,\eps',y),\;\;s(\eps,\eps',x)=s(\eps,\eps',y).
\end{equation}
Furthermore, if $\rho(x)\geq 2\eps$ then $t(2\eps,x)\leq 0$ and so $\beta(\eps,\eps',x)\geq 1$. From the properties of $h$ we thus have $\eta(\eps,\eps',x)=0$ and $\bar f(\eps,\eps',x)=x$. Similarly $\eta(\eps,\eps',x)=0$ and $\bar f (\eps,\eps',x)=x$ if $\rho(x)\leq 2\eps'$.

To prove that $\bar f(\eps,\eps',\cdot)$ is one-to-one, suppose that $\bar f(\eps,\eps',x)=\bar f(\eps,\eps',y)$. If both $\rho(x), \rho(y)\not \in (2\eps',2\eps)$ then clearly $x=y$. If $\rho(x)\not\in (2\eps',2\eps)$ and $\rho(y)\in (2\eps',2\eps)$ then $t(\eps,y)\in(t(2\eps',y),t(2\eps,y))$. Hence $\beta(\eps,\eps',y)\in(0,1)$, thus $\alpha(\eps,\eps',y)\eta(\eps,\eps',y)+\beta(\eps,\eps',y)\in(0,1)$, from which it follows that  $\eta(\eps,\eps',y)\in(t(2\eps',y),t(2\eps,y))$. Hence $\rho(\bar f(\eps,\eps',y))\in (2\eps',2\eps)$ and so $\bar f(\eps,\eps',y)\neq x=\bar f(\eps,\eps',x)$. So this case cannot occur.

If both $\rho(x), \rho(y)\in (2\eps',2\eps)$ then $y=S(\sigma)x$, where $\sigma=\eta(\eps,\eps',x)-\eta(\eps,\eps',y)$. Let $\gamma= \rho(\bar f(\eps,\eps',x))=\rho(\bar f(\eps,\eps',y))$. Then  $\eta(\eps,\eps',x) = t(\gamma,x), \eta(\eps,\eps', y)=t(\gamma, y)$. Thus  by \eqref{timetranslate}, \eqref{alphabeta1}  $$\alpha(\eps,\eps',x)\eta(\eps,\eps',x)+\beta(\eps,\eps',x)=\alpha(\eps,\eps',y)\eta(\eps,\eps',y)+\beta(\eps,\eps',y).$$   But from \eqref{sub1}, \eqref{rs} this implies that 
$$h(r(\eps,\eps',x), s(\eps,\eps',x), \beta(\eps,\eps',x))=h(r(\eps,\eps',x), s(\eps,\eps',x), \beta(\eps,\eps',y)),$$
which implies by the strict monotonicity of $h(a,c,t)$ in $t$  that $\beta(\eps,\eps',x)=\beta(\eps,\eps',y)$ and hence $\sigma =0$ and $x=y$.

To prove that $\bar f(\eps,\eps',\cdot)$ is onto it suffices to show that if $\rho(y)\in (3\eps',3\eps)$ then $S(\eta(\eps,\eps',x))x=y$ for some $x$, and we claim that such an $x$ with $\rho(x)\in (3\eps',3\eps)$ is given by
\begin{equation}\label{fbarinverse}
x=S(\tau(\eps,\eps',y))y;\;\;\tau(\eps,\eps',y)=\alpha(\eps,\eps',y)^{-1}[h^{-1}(r(\varepsilon,\varepsilon',y),s(\varepsilon,\varepsilon',y),\beta(\eps,\eps',y))- \beta(\eps,\eps',y)],
\end{equation}
where $h^{-1}(a,c,\cdot)$ denotes the inverse function of $h(a,c,\cdot)$. To show that $\rho(x)\in(3\eps',3\eps)$ it suffices to prove that $\tau(\eps,\eps',y)\in(t(3\eps',y),t(3\eps,y))$. But this holds because $\alpha(\eps,\eps',y)t(3\eps,y)+\beta(\eps,\eps',y)>1$, $\alpha(\eps,\eps',y)t(3\eps',y)+\beta(\eps,\eps',y)<0$, using  $h_t(a,c,t)>0$  and the fact that $h(a,c,t)=t$ for $t\not \in(0,1)$. For $x$ given by \eqref{fbarinverse} we deduce from \eqref{alphabeta1} that 
\begin{equation}
\nonumber
\beta(\eps,\eps',x)=\beta(\eps,\eps',y)+\tau(\eps,\eps',y)\alpha(\eps,\eps',y)= h^{-1}(r(\varepsilon,\varepsilon',y),s(\varepsilon,\varepsilon',y),\beta(\eps,\eps',y)),
\end{equation}
so that 
\begin{equation}\nonumber
\eta(\eps,\eps',x)=\alpha^{-1}(\eps,\eps',y)[\beta(\eps,\eps',y)-\beta(\eps,\eps',x)]=-\tau(\eps,\eps',y)
\end{equation}
 as required. 

Next we note that $\rho(x)>\eps$ if and only if $t(\eps,x)<0$, which holds if and only if
$$h(\alpha(\eps,\eps',x)t(\eps,x)+\beta(\eps,\eps',x),\alpha(\eps,\eps',x)t(\eps',x)+\beta(\eps,\eps',x),\beta(\eps,\eps',x))>\alpha(\eps,\eps',x)t(\eps',x)+\beta(\eps,\eps',x),$$
since $h(a,c,a)=c$ and $h_t(a,c,t)>0$, thus if and only if $\rho(S(\eta(\eps,\eps',x))x)>\eps'$. Thus $\bar f(\eps,\eps',\Omega_\eps)=\Omega_{\eps'}$ and  $\bar f(\eps,\eps',\partial\Omega_\eps)=\partial\Omega_{\eps'}$. That $\bar f, \bar f^{-1}$ are smooth functions of $(\eps,\eps',x)\in (0,\eps_0)\times(-\eps_0,0)\times \R^m$ follows from  the smoothness of $\alpha,\beta,r,s,h$ and $h^{-1}$ (the latter by $h_t>0$ and the inverse function theorem) together with \eqref{fbarinverse}. This completes the proof of (iii).

To complete the proof of the theorem, let us temporarily denote the map $f:[0,\eps_0)\times \bar\Omega\to\R^m$ constructed at the beginning of the proof by $\tilde f$. We need to extend $\tilde f$ to a map $f:(-\eps_0,\eps_0)\times\R^m\to\R^m$ satisfying (ii). We define $f=f(\eps,x)$ by $f(0,x)=x$ and 
\begin{equation}\label{fdefn}
f(\eps,x)=\left\{\begin{array}{ll}\tilde f(\eps,x)&\mbox{if }\eps\in(0,\eps_0), x\in \bar\Omega,\\
\bar f^{-1}(\eps,-\eps,F(-\eps,x))&\mbox{if }\eps\in(0,\eps_0), x\in \Omega^c,\\
\bar f(-\eps,\eps,\tilde f(-\eps,x))&\mbox{if }\eps\in(-\eps_0,0), x\in\bar\Omega,\\
F(\eps,x)&\mbox{if }\eps\in(-\eps_0,0),x\in \Omega^c.
\end{array}\right.
\end{equation}
Note that the domains of definition of $f(\eps,x)$ overlap for $x\in\partial\Omega$. However the definitions coincide there because by construction in each case $f(\eps,x)$ lies on the intersection of the orbit of the flow of good directions through $x$ with $\partial\Omega_\eps$, and this point is unique. The properties of $\tilde f, \bar f$ and $F$ imply that $f(\eps,\cdot)$ is a homeomorphism of $\R^m$ onto $\R^m$ with $f(\eps,\Omega)=\Omega_\eps, f(\eps,\partial\Omega)=\partial\Omega_\eps$, with inverse given by
 $f^{-1}(0,x)=x$ and 
\begin{equation}\label{finverse}
f^{-1}(\eps,x)=\left\{\begin{array}{ll}\tilde f^{-1}(\eps,x)&\mbox{if }\eps\in(0,\eps_o), x\in\bar\Omega_\eps,\\
F^{-1}(-\eps,\bar f(\eps,-\eps,x))&\mbox{if }\eps\in(0,\eps_0), x\in \Omega^c_\eps,\\
\tilde f^{-1}(-\eps,\bar f^{-1}(-\eps,\eps,x))&\mbox{if }\eps\in (-\eps_0,0), x\in\bar\Omega_\eps,\\
F^{-1}(\eps,x)&\mbox{if }\eps\in (-\eps_0,0), x\in\Omega^c_\eps.
\end{array}\right.
\end{equation}
The continuity of $f(\eps,x)$ and $f^{-1}(\eps,x)$  for $0\leq|\eps|<\eps_0, x\in \R^m$ follows since, as is easily checked, $f(\eps,-\eps,x)\to x, f^{-1}(\eps,-\eps,x)\to x$ uniformly as $\eps\to 0+$. The smoothness of $f(\eps,x)$ for $0<|\eps|<\eps_0, x\not \in\partial\Omega$, and of $f^{-1}(\eps,x)$ for $0<|\eps|<\eps_0, x\not \in\partial\Omega_\eps$, follows from the corresponding properties of $\tilde f, \bar f$ and $F$. Finally, by construction $f(\eps,x)=x$ for $|\rho(x)|>3|\eps|$.

This completes the proof.
\end{proof}
\begin{remark}
\label{smoothclosures}\rm
The above proof uses part (iii) of the theorem to help prove part (ii). Conversely, given (ii), for $-\eps_0<\eps'<0<\eps<\eps_0$ we can define $f(\eps,\eps',x)=f(\eps',f^{-1}(\eps,x))$, which is a homeomorphism of $\R^m$ onto $\R^m$ such that $f(\eps,\eps',\Omega_\eps)=\Omega_{\eps'}$, $f(\eps,\eps',\partial\Omega_\eps)=\partial\Omega_{\eps'}$, $f(\eps,\eps',\cdot):\Omega_\eps\to\Omega_{\eps'}$ is a diffeomorphism, and $f(\eps,\eps', x)=x$ if $|\rho(x)|>\max (-3\eps',3\eps)$. However (iii) gives extra information, in particular that $f(\eps,\eps',\cdot)$ is a diffeomorphism of $\bar\Omega_\eps$ onto $\bar\Omega_{\eps'}$.
\end{remark}
\begin{remark}\label{genapprox}{\rm 
As observed by Fraenkel \cite[Section 5]{fraenkelc0} the image under a diffeomorphism of a bounded domain of class $C^0$ need not be a domain of class $C^0$, since a cusp such as in Remark \ref{cusp}   can be bent by the diffeomorphism so that the boundary is not locally a graph. However Theorem \ref{homapprox} immediately implies a corresponding smooth approximation result for the larger class of bounded domains $\Omega\subset\R^m$ which are the image under a $C^\infty$ diffeomorphism $\varphi:U\rightarrow\R^m$ of a bounded domain $\Omega'\subset\R^m$ of class $C^0$, where $U$ is an open neighbourhood of $\overline{\Omega'}$. If $\Omega_\varepsilon', 0<|\varepsilon|<\varepsilon_0,$ are the approximating domains given by the theorem for $\Omega'$ for $\varepsilon_0>0$ sufficiently small, then the open sets $\Omega_\varepsilon=\varphi(\Omega_\varepsilon')$ are a family of bounded domains of class $C^\infty$ such that $\bigcap_{-\varepsilon_0<\varepsilon<0}\Omega_\varepsilon=\bar\Omega, \bigcup_{0<\varepsilon<\varepsilon_0}\Omega_\varepsilon=\Omega$.}
\end{remark}
 
\begin{remark}\label{topman}\rm If $\Omega$ is Lipschitz, then the homeomorphism between $\overline\Omega_\varepsilon$ and $\overline\Omega$ defined in the 
proof of Theorem~\ref{homapprox} is a bi-Lipschitz map (with Lipschitz constants bounded independently of $\varepsilon$ for $0<|\varepsilon|<\varepsilon_0$). In order to check this it suffices to show that the functions $t$ and $\beta$ (for the interior approximation, say)  are Lipschitz. This can be seen in the case of $t$, for example,  by applying $\frac{\partial}{\partial x_i}$ to $F(t(x),x)=0$ with $F$ as in   \eqref{eq:Fimplicit},  obtaining thus  that $\frac{\partial t}{\partial x_i}=- \frac{\partial F}{\partial x_i}/\frac{\partial F}{\partial t}$. Given $x_0\in\partial\Omega$ with a corresponding good direction $n$, there exist $\delta=\delta(x_0)$ and $c_0(\delta), c_1(\delta)$ such that $0<c_0(\delta)<\nabla\rho(x)\cdot n,\; |\nabla\rho(x)|<c_1(\delta)$ for all $x\in B(x_0,\delta)\cap\Omega$  (see 
\cite[p. 63, relation (A.3), Lemma A.1 and the line after (A.7)]{lieberman2}). We can then apply compactness and Remark \ref{remark:nonnegative angle} to show that for some $\delta_1>0$ and constants $c_1, c_2$ depending only on $\delta_1$ we have $0<c_0<\nabla\rho(x)\cdot G(x),  |\nabla\rho(x)|\leq c_1$ for all $x\in\Omega$ with ${\rm dist}\,(x,\partial\Omega)<\delta_1$. Hence $\frac{\partial F(t(x),x)}{\partial t}$ is bounded away from zero, and $\frac{\partial F(t(x),x)}{\partial x_i}$ is bounded, for all $x\in \Omega$ with ${\rm dist}\,(x,\partial\Omega)<\delta_1$. Thus $|\nabla f(x)|$ is bounded for such $x$ and hence for all $x\in \Omega$. Hence $f\in W^{1,\infty}(\Omega,\R^m)$. Since $\Omega$ is Lipschitz this implies that $f$ is Lipschitz. This follows, for example, by noting that by Stein \cite[Chapter VI, Theorem 5]{stein} $f$ may be extended to a function $\tilde f\in W^{1,\infty}(\R^m,\R^m)$ so that $\tilde f$ is Lipschitz (a more general result can be found in \cite[Theorem 4.1]{heinonen}). 

  Related results concerning the approximation of Lipschitz domains were obtained in \cite{necas-paper, verchota-thesis, verchota-JFA} \rm (for the last two papers see the comments in the introduction). The recent paper of Amrouche, Ciarlet \& Mardare \cite{amroucheetal}, which refers to an earlier version of our paper, contains a statement (their Theorem 2.2) of a special case of Theorem \ref{homapprox} for $\Omega$ Lipschitz.
\end{remark}
\begin{remark}\label{topman1}\rm 
Theorem \ref{homapprox} implies in particular that if $\Omega\subset\R^m$ is a bounded domain of class $C^0$, then $\bar\Omega$ and $\Omega^c$ have differential structures making them $C^\infty$ $m$-manifolds with boundary, and $\partial\Omega$ has a differential structure making it a $C^\infty$ $(m-1)-$manifold, since $\bar\Omega$ (resp. $\Omega^c$, $\partial\Omega$) is homeomorphic to $\bar\Omega_\eps$ (resp. $\Omega_\eps^c$, $\partial\Omega$) for $0<|\eps|<\eps_0$, which is such a $C^\infty$ manifold. The field of good directions can be thought of as a transverse field to $\partial\Omega$, and more generally the existence of a transverse field to a topological manifold embedded in $\R^m$ is known to imply the existence of a smooth differential structure following the work of Cairns \cite{cairns} and Whitehead \cite{whitehead} (see also Pugh \cite{Pugh}). 

A referee has  drawn our attention to the possibility of using the techniques of the theory of smoothing of manifolds to prove the existence of diffeomorphic interior and exterior approximations by $C^\infty$  domains for the  wider class of bounded domains $\Omega\subset\R^m$ whose boundaries are locally flat in the sense of Brown \cite{brown}, that is $\bar\Omega$ and $\Omega^c$ are topological manifolds with boundary (one advantage of this class of domains is that it is invariant under homeomorphisms). Indeed this seems likely to be the case for $m\geq 6$ making use of Kirby-Siebenmann theory \cite{Kirby-essay}, and for $m=2,3$ using the classical results on smoothing 2- and 3-manifolds of   Rad\'o \cite{rado} and Moise \cite{moise} (see Hamilton \cite{hamilton} and Hatcher \cite{hatcher1} for different treatments). However, we are not aware of explicit results of this type in the literature. We do not pursue this interesting direction further here. While applying to a more restricted set of domains, Theorem \ref{homapprox} is valid in all dimensions, with a relatively simple dimension-independent proof, and gives a semi-explicit representation of the approximating domains together with a potentially useful technique for varying domains through the flow of good directions. 

 Theorem \ref{homapprox} gives the existence of a homeomorphism $f(\eps,\cdot)$ of $\bar\Omega$ onto $\bar\Omega_\eps$ for $0<|\eps|<\eps_0$ that is a diffeomorphism of $\Omega$ onto $\Omega_\eps$. In general, if $U\subset \R^m, V\subset\R^m$ are homeomorphic open sets, then if $m=1,2,3$ it is known that $U$ and $V$ are diffeomorphic, while if $m\geq 5$ and $U$ is the whole of $\R^m$ (or equivalently $U$ is an open ball) then $U$ and $V$ are also diffeomorphic, since in these cases $U$ has a unique differential structure up to diffeomorphism. These results are due to \cite{moise}, \cite{munkres}, \cite{stallings} and are surveyed in \cite{milnor2011}. This is not true if $m=4$ because of the existence (following from the work of Freedman \cite{freedman} and Donaldson \cite{donaldson}) of `small exotic $\R^4$s', which implies that there is a bounded open subset of $\R^4$ which is homeomorphic but not diffeomorphic to an open ball in $\R^4$;  for discussions of this work see \cite{scorpan}, and  \cite[Chapter XIV]{kirby}. Whether two arbitrary homeomorphic open subsets of $\R^m$, $m\geq 5$, are diffeomorphic seems not to be known in general. 
\end{remark}

\begin{remark}\label{remark:arbitrary omega} \rm For an arbitrary open set $\Omega\subset\R^m$ with  $\bar\Omega\neq\R^m$ one can easily find smooth sets approximating it and contained in $\Omega$ (respectively $\R^m\setminus\bar\Omega$), provided that one does not require the approximating sets to preserve the topology or diffeomorphism class  of $\Omega$. For instance, as shown in \cite[Chapter 6]{stein} one can always find a regularised distance $\tilde\rho$ and then for almost all $\varepsilon$ small enough the sets $\Omega_\varepsilon$ defined as in \eqref{omegaep} (but with $\rho$ replaced by $\tilde\rho$) will be of class $C^\infty$ (due to Sard's theorem, see for instance \cite[p. 35]{mazya}).
\end{remark}

We now discuss counterexamples to the conclusions of Theorem \ref{homapprox} when $\Omega$ is not of class $C^0$. It is easy to construct such examples for the exterior approximation.
\begin{example}
\label{exteriorex}\rm
Let $\Omega\subset\R^2$ be the simply-connected domain defined by $$\Omega=B(0,1)\setminus\overline{ B((0,1/2)),1/2)}.$$

Then $\Omega$ is not of class $C^0$ because the boundary cannot be represented as a graph in the neighbourhood of the boundary point $(0,1)$. We claim that there is no decreasing sequence of bounded domains $\Omega^{(j)}\subset\R^2, j=1,2,\ldots,$ which are each homeomorphic to $\Omega$ and such that $\bar \Omega=\bigcap_{j=1}^\infty\Omega^{(j)}$. Indeed, for any such sequence we have that $S^1=\partial B(0,1)\subset \Omega^{(j)}$ for all $j$, but $(0,1/2)\not\in\Omega^{(j)}$ for sufficiently large $j$, so that $\Omega^{(j)}$ is not simply-connected for sufficiently large $j$.
\end{example}

In order to give  counterexamples for the interior approximation we will consider bounded open subsets   $\Omega\subset\R^m$ which are topological manifolds with boundary (strictly speaking, such that  $\bar\Omega$ is a topological manifold with boundary), i.e. each point $x\in\partial \Omega$ has a neighbourhood in $\bar \Omega$ that can be mapped homeomorphically onto a relatively open subset of $H^m=\{(x_1,\ldots,x_m)\in\R^m: x_m\geq 0\}$. Any bounded domain of class $C^0$ is a topological manifold with boundary (see, for example, \cite[Appendix A]{ballzarnescu}), but not conversely; for example, the interior of a Jordan curve in $\R^2$ is a topological manifold with boundary by the Schoenflies theorem.
\begin{proposition}
Let $\Omega\subset\R^m$ be a bounded open set   that is a topological manifold with boundary. Let $\Omega'\subset\R^m$ be a bounded open set that is homeomorphic to $\Omega$. Then $\partial\Omega$ has a finite number $N$ of boundary components, and $\partial\Omega'$ has a finite number $N'$ of boundary components, where $N'\leq N$.
\end{proposition}
\begin{proof}
We first claim that $\partial\Omega$ has a finite number $N$ of boundary components $\Gamma_i, 1\leq i\leq N$. If there were infinitely many then there would be a point $\xi\in\partial\Omega$ and a sequence $\xi_i\to \xi$ with $\xi_i\in \Gamma_i, i=1,2,\ldots,\infty$.  Since $\Omega$ is a topological manifold with boundary, there   exist $\delta>0$ and a homeomorphism $\psi: B(\xi,\delta)\cap\bar\Omega\to V$, where $V$ is a relatively open subset of $H^m$. Then for $\bar\delta>0$ sufficiently small, $\psi^{-1}(B(\psi(\xi),\bar\delta)\cap\partial H^m)$ is a connected open neighbourhood of $\xi$ in $\partial\Omega$ to which $\xi_i$ belongs for sufficiently large $i$, a contradiction.

By the collaring theorem of Brown \cite{brown} (for an alternative proof see \cite{connelly}),  for each $i=1,\ldots, N$ there is a homeomorphism $\psi_i$ mapping  $\Gamma_i\times [0,1)$ onto a relatively open neighbourhood $U_i$ of $\Gamma_i$ in $\bar\Omega$ such that $\psi_i(x,0)=x$ for all $x\in\Gamma_i$. Consider the open subset $U_{i,\varepsilon}=\psi_i(\Gamma_i\times (0,\varepsilon))$ of $\Omega$. Since the product of connected sets is connected, and $\psi_i$ is continuous, $U_{i,\varepsilon}$ is connected.

 By assumption there is a  homeomorphism $\varphi:\Omega\to\Omega'$. Therefore $\varphi(U_{i,\varepsilon})$ is connected. Hence $\overline{\varphi(U_{i,\varepsilon})}$ is a closed connected subset of $\overline{\Omega'}$. Let $\varepsilon_j\to 0$. We may assume that the sets $\overline{\varphi(U_{i,\varepsilon_j})}$ converge to a subset $V_i$ of $\overline{\Omega'}$ in the Hausdorff metric, and $V_i$ is connected. We claim that $V_i\subset\partial\Omega'$. Indeed if $z\in \Omega'\cap V_i$ there would exist  a sequence $x_k\in\Omega$ with $x_k\to\bar x\in\Gamma_i$ and $\varphi(x_k)\to z$. But then $x_k\to \varphi^{-1}(z)\in\Omega$, a contradiction.  Hence for each $i$ we have a corresponding connected subset $V_i$ of $\partial\Omega'$. We claim that $\partial\Omega'\subset\bigcup_{i=1}^NV_i$. Indeed suppose that $z\in \partial\Omega'$. Then there exists a sequence $z_k\in \Omega'$ with $z_k\to z$. Then $\varphi^{-1}(z_k)\in\Omega$ for each $k$, and we may assume that $\varphi^{-1}(z_k)\to x\in\bar\Omega$. If $x\in\Omega$ then $z_k\to\varphi(x)=z\in\Omega'$, a contradiction. Hence $\varphi^{-1}(z_k)\to x\in\partial\Omega$, and so $x\in\Gamma_i$ for some $i$. So there exists a subsequence $z_{k_j}$ with $\varphi^{-1}(z_{k_j})\in U_{i,\varepsilon_j}$, and hence $z\in V_i$. Thus $\partial\Omega'$ has $N'$ components for some $N'\leq N$.
\end{proof}
\begin{remark}\rm In general $N'<N$. For example if $\Omega$ consists of the union of two balls with disjoint closures while $\Omega'$ consists of two disjoint open balls whose boundaries touch, then we have $N=2$ and $N'=1$. It is perhaps true that $N'=N$ if $\Omega$ is connected, but we do not need this.
\end{remark}
\begin{example}
\label{interiorex}\rm Let 
$$\Omega=B(0,1)\setminus \bigcup_{j=1}^\infty \overline{B((1-{1}/{j})e_1,1/{3j^2})}\subset\R^m,$$
where $e_1=(1,0,\ldots,0)$. Then $\Omega$ is a bounded domain, but $\partial\Omega$ has infinitely many boundary components, and thus $\Omega$ cannot be homeomorphic to a bounded domain of class $C^\infty$ (or of class $C^0$).
\end{example}

\section{The topology  of $\Omega$  and the properties of the map of good directions}

\medskip In Section~\ref{good} we constructed special smooth fields of good directions, that we called canonical. In this section we study the properties of arbitrary  continuous fields of good directions that are not necessarily canonical.
\par We start with an illustrative case that   provides significant insight into more general situations. We consider a standard solid torus in $\R^3$ given by $\Omega_T=T([0,2\pi]\times [0,2\pi]\times [0,1))$, where
 $$T:[0,2\pi]\times [0,2\pi]\times [0,1)\to\R^3,\quad T(\theta,\varphi,r)=(\cos\theta\,(2+r\cos\varphi),\sin\theta\,(2+r\cos\varphi), r\sin\varphi),$$ whose boundary is  $\partial\Omega_T=\To\defeq T_b([0,2\pi]^2)$, where
\be\label{def:T_b}
T_b:[0,2\pi]\times [0,2\pi]\to\R^3, \quad T_b(\theta,\varphi)=(\cos\theta\,(2+\cos\varphi),\sin\theta\,(2+\cos\varphi), \sin\varphi).
\ee
A continuous field $C:\To\to\Sph^2$  is a field of pseudonormals with respect to $\Omega_T$ if and only if  
\be\label{cond:geom}
C(P)\cdot \nu_{\To} (P)>0\textrm{ for all }P\in \To, 
\ee where $\nu_{\To}(P)$ denotes the interior normal to $\Omega_T$ at $P\in\To$. The geometrical condition \eqref{cond:geom} imposes a constraint on the image of the field $C$:
\begin{proposition}\label{prop:goodtorus}
 For any continuous field  $C:\To\to\Sph^2$ satisfying the geometrical condition \eqref{cond:geom}  there exists a band  of size $2\delta$ around the equator on the unit sphere, namely 
\be 
E_\delta\defeq\{ n\in\Sph^2:  |n \cdot e_3| <\delta\},
\ee such that $E_\delta\subset C(\To)$, where $e_3=(0,0,1)$.
 
 Conversely, for any given $\gamma>0$ there exists a continuous field $C_{\gamma}:\To\to\Sph^2$ satisfying  \eqref{cond:geom}   such that $C_{\gamma}(\To)\subset E_{\gamma}$.
\end{proposition}

\begin{proof}
In order to prove the first claim we use degree theory for the map $C$, on a domain $\omega\subset\To$ with respect to the point $n\in\Sph^2$, denoted $d(C,\omega,n)$. Since the domain $\omega$ we use is diffeomorphic to a bounded open subset of $\R^2$ we can use the theory of degree for subsets of Euclidean space as described, for example, in \cite{fonsecadegree}. It suffices to   show that $d(C,\omega,n)=1$ for all $n\in E_\delta$ with suitable $\delta>0$ and $\omega\subset\To$ (because this implies that $E_\delta\subset C(\omega)$). In order to show this we use the homotopy invariance of the degree \cite[Theorem 2.3]{fonsecadegree} for the homotopy 
$H:[0,1]\times\omega\to\Sph^2$ defined by  
$$H(\lambda,P)=\frac{\lambda C(P)+(1-\lambda)\nu_{\To}(P)}{|\lambda C(P)+(1-\lambda)\nu_{\To}(P)|}$$ 
that connects the smooth negative Gauss map $\nu_{\To}$ with the field $C$. 
We choose 
$$\omega\stackrel{\rm def}{=}T_b([0,2\pi]\times\{(0,\pi/2)\cup(3\pi/2,2\pi]\})$$ 
to be the `exterior part' of the torus. We first note that $d(\nu_{\To},\omega,n)=1$ for all $n\in\Sph^2\setminus \{\pm e_3\}$. This is because the Jacobian of $\nu_{\To}$ equals the Gaussian curvature of the torus which is positive in $\omega$, and because such $n$ do not belong to $\nu_{\To}(\partial\omega)$ and have exactly one inverse image  in $\omega$ under $\nu_{\To}$. Next observe that the condition $C(P)\cdot \nu_{\To}(P)>0$ and the continuity of $C$ ensures that there exists a $\delta>0$ so that $C(P)\not\in E_\delta$ for $P\in\partial\omega$. Thus for $n\in E_\delta$ the condition $n\not\in \{H(\lambda,P),\lambda\in [0,1],P\in\partial\omega\}$ is satisfied and we can apply the homotopy invariance of the degree to conclude that  $d(C,\omega,n)=1$ for $n\in E_\delta$, and  hence $E_\delta\subset C(\To)$.  

\bigskip

  To prove the second claim of the proposition we follow a suggestion of an anonymous referee, which provides a simpler example than in our original approach. We   denote by $\tilde n(\theta,\varphi)\in\Sph^2$ the interior normal at $T_b(\theta,\varphi)$ (where $T_b$ was defined in \eqref{def:T_b}). We let $t:[0,2\pi]\to \Sph^2$ be given by $t(\theta):=(-\sin\theta,\cos\theta,0)$. Then $t(\theta)\cdot \tilde n(\theta,\varphi)=0$ for all $(\theta,\varphi)\in [0,2\pi]^2$. Defining $p_\varepsilon:[0,2\pi]^2\to\Sph^2$  for $\eps\in(0,1)$ by
  
  $$p_\varepsilon(\theta,\varphi):=\frac{(1-\varepsilon)t(\theta)+\varepsilon \tilde n(\theta,\varphi)}{\sqrt{1-2\varepsilon+2\varepsilon^2}}$$
    we have $|p_\varepsilon(\theta,\varphi)|=1$ for all $(\theta,\varphi)\in [0,2\pi]^2$ and $p_\varepsilon\cdot \tilde n=\frac{\varepsilon}{\sqrt{1-2\eps+\eps^2}}$,  so that $p_\varepsilon$ is a pseudonormal field. On the other hand $|p_\varepsilon\cdot e_3|=\left|\frac{\varepsilon\sin\varphi}{\sqrt{1-2\varepsilon+2\varepsilon^2}}\right|\leq\frac{\varepsilon}{\sqrt{1-2\varepsilon+2\varepsilon^2}}$.  Hence for any $\gamma<1$ there exists $\varepsilon$ such that  $p_\varepsilon (\To)\subset E_\gamma$.
  \end{proof}
 We now explore what happens for more general bounded domains.

\begin{theorem}\label{prop:surjectivity}
Let $\Omega\subset\R^m$ be a bounded domain of class $C^0$.
\smallskip\par {\rm (i)} If $m\ge 3$,  and the Euler characteristic of $\Omega$ is non-zero then any  continuous pseudonormal field $n:\partial\Omega\to\Sph^{m-1}$ is surjective.
 \par If  $m=3$ and  the Euler characteristic of $\Omega$ is zero then a continuous pseudonormal field is not necessarily surjective.\label{lemma:surjectivity}
 \smallskip\par {\rm (ii)} If $m=2$ any continuous pseudonormal field $n:\partial\Omega\to\Sph^1$ is surjective.
\end{theorem}

\begin{proof} (i) We first prove that if the Euler characteristic of $\Omega$ is non-zero then $n:\partial\Omega \to\Sph^{m-1}$ is surjective. We assume for contradiction that $n$ is not surjective. Thus there exists a  $\delta>0$ small enough so that any continuous field $\bar n:\partial\Omega\to\Sph^{m-1}$ such that
$\|\bar n-n\|_{C(\partial\Omega)}<\delta$ is also not surjective. We claim that there exists a continuous field of good directions  $\tilde n$ defined on a neighbourhood $V$ of $\partial\Omega$ so that $\|\tilde n-n\|_{C(\partial\Omega)}<\delta/2$.
 \par In order to prove the claim we first note that as $n:\partial\Omega\to\Sph^{m-1}$ is continuous and $\partial\Omega$ is bounded, there exists  $\delta'>0$  such that
 \be
 |n(x)-n(y)|<\frac{\delta}{4}\textrm{ for all } x,y\in\partial\Omega \textrm { with } |x-y|<\delta'.
 \label{ncont}
 \ee
\par As $\Omega$ is of class $C^0$ there exist some $\bar\delta>0$ and points $P_1,\dots, P_l\in\partial\Omega$ such that $\partial\Omega\subset\cup_{j=1}^l B_{\bar\delta}(P_j)$, and such that for any $j=1,2,\dots,l$ and $R\in B_{2\bar\delta (P_j)}$ we have that $n(P_j)$ is a good direction with respect to $\Omega$ at $R$. We assume, without loss of generality, that $8\bar\delta<\delta'$.
 \par Recalling the definition \eqref{omegaep}  of the sets $\Omega_\eps$ and from \eqref{equivdist} that  $\frac{1}{2}\le\frac{\rho(x)}{d(x)}\le 2\textrm{ for all } x\in\R^m$, we have that $\Omega_{-\bar\delta/2}\setminus \Omega_{\bar\delta/2}\subset \cup_{j=1}^{l}B_{2\bar\delta}(P_j)$. Consider a partition of unity $\alpha_j,j=1,\dots, l$, such that $\alpha_j\in C_0^\infty(B_{4\bar\delta}(P_j))$ and $\Sigma_{j=1}^l
\alpha_j(x)=1,\textrm{ for all } x\in\cup_{j=1}^l B_{2\bar\delta}(P_j)$.  Let
$$\hat n(P)\defeq\sum_{j=1}^l\alpha_j(P) n(P_j).$$
Then, by Lemma \ref{ggd}, $\hat n(P)\neq 0$ and $\frac{\hat n(P)}{|\hat n(P)|}$ is a good direction at $P$ for any $P\in \Omega_{-\bar\delta/2}\setminus \Omega_{\bar\delta/2}$.  
 Taking into account that $\alpha_j(P)=0$ for $|P-P_j|>4\bar\delta$,  that $4\bar\delta<\delta'$, and \eqref{ncont} we obtain
\be\label{nnhat}
|n(P)-\hat n(P)|\le \sum_{j=1}^l \alpha_j(P)|n(P)- n(P_j)| \le \frac{\delta}{4}\textrm{ for all }P\in\partial\Omega.
\ee 
On the other hand, for $P\in \partial\Omega$ we have
\begin{align}
\left|\hat n(P)-\frac{\hat n(P)}{|\hat n(P)|}\right|=||\hat n(P)|-1|\le \left|\sum_{j=1}^l \alpha_j(P)n(P_j)-\sum_{j=1}^l\alpha_j(P)n(R)\right|\nonumber\\
=\left|\sum_{j\in\mathcal{J}_P}\alpha_j(P)(n(P_j)-n(R))\right|\le \sum_{j\in\mathcal{J}_P}\alpha_j(P)|n(P_j)-n(R)|<\frac{\delta}{4},\label{nnhat1}
\end{align} where $\mathcal{J}_P\stackrel{\rm def}{=}\{j\in\{1,2,\dots,l\}; |P-P_j|\le 4\bar\delta\}$ and $R=P_i$ for some arbitrary $i\in\mathcal{J}_P$. For the last inequality we used that $|R-P_j|\le 8\bar\delta\le\delta',\textrm{ for all } j\in\mathcal{J}_P$ and \eqref{ncont} together with  $\sum_{j\in\mathcal{J}_P}\alpha_j(P)=1$.
We take $\tilde n(P)\stackrel{\rm def}{=}\frac{\hat n(P)}{|\hat n(P)|}$ on  $V\defeq\Omega_{-\bar\delta/2}\setminus\Omega_{\bar\delta/2}$ and \eqref{nnhat}, \eqref{nnhat1}   prove our claim about the existence of $\tilde n$. 
\par We denote by $n_\eps(P)$   the interior normal to $\partial\Omega_\eps$ at $P$, so that $-n_\eps:\partial\Omega_\eps\to \Sph^{m-1}$ is the Gauss map; note that  $n_\eps$ is parallel to $\nabla\rho(P)$ and has the same degree as $-n_\eps$  up to (possible) change of sign. As noted above $\tilde n|_{\partial\Omega}$ is not surjective. 
Hence, as $\tilde n$ is continuous on $V$, there exists  $\eps_1>0$ so that $\Omega_{-\eps_1}\setminus\Omega_{\eps_1}\subset V$ and $\tilde n|_{\overline{\Omega_{-\eps_1}\setminus\Omega_{\eps_1}}}$ is also not surjective.
 Moreover, by Theorem~\ref{homapprox} the sets $\Omega_\eps$ and $\Omega$ are homeomorphic and thus they have the same Euler characteristic \cite{hatcher}. On the other hand for the smooth domain $\Omega_\eps$ the Euler characteristic equals the degree of the Gauss map (\cite[p. 384]{bredon}) and hence the Gauss map has non-zero degree. For any $P\in\partial\Omega_\eps$ we have that both $\tilde n(P)$ and $n_\eps(P)$ are good directions at $P$, so that by Lemma \ref{geoconv} we have that   $\tilde n(P)\cdot n_\eps(P)\not=-1$. The  homotopy $h:[0,1]\times\partial\Omega_\eps\to\Sph^{m-1}$ connecting $\tilde n|_{\partial\Omega_\eps}$ and $n_\eps$ given by
 $$h(t,P)=\frac{t\tilde n(P)+(1-t)n_\eps (P)}{|t\tilde n(P)+(1-t)n_\eps (P)|}$$
is thus well defined. Hence $n_\eps$  has the same non-zero degree as $\tilde n|_{\partial\Omega_\eps}$ and thus $\tilde n|_{\partial\Omega_\eps}$ is surjective (see \cite[pp.123,125]{hirsch}), a contradiction. Hence $n$ is surjective.

\smallskip\par If $m=3$ and the Euler characteristic of $\Omega$ is zero  the second part of Proposition~\ref{prop:goodtorus} provides the required counterexample.

\smallskip\par (ii)   In the same way as for part (i) it suffices to show that $\tilde n|_{\partial\Omega_\eps}$  is surjective for nonzero $|\varepsilon|$ sufficiently small. We  observe that for $\eps\not=0$ the set $\Omega_\eps$ is a smooth 2D manifold with boundary. By the classification theorem for 1D connected, compact smooth manifolds (see for instance \cite{milnor}) we have that each connected component $\Gamma$ of $\partial\Omega_\eps$ is diffeomorphic to $\Sph^1$. Thus $\Gamma$ can be parametrized as a smooth closed curve $\gamma:\Sph^1\to \R^2$ with constant speed $|\dot\gamma(t)|=s>0$. Let $N(t)=n_\varepsilon(\gamma(t))$ for $t\in\Sph^1$. Then  $\Delta(t)=\dot\gamma_1(t)N_2(t)-\dot\gamma_2(t)N_1(t)$ equals $\pm s$ for each $t\in \Sph^1$, and since $\Delta(t)$ is continuous the sign of $\Delta(t)$ is independent of $t\in\Sph^1$.   The  {\it Umlaufsatz} theorem of Hopf (see for instance \cite[p. 275]{chavel}) guarantees that  $\dot{\gamma}(t)/|\dot\gamma(t)|$ has degree $\pm 1$  when regarded as a function from $\Sph^1$ into  $\Sph^1$. The fact that $\Delta(t)$ has constant sign implies that    $N:\Sph^1\to\Sph^1$ is homotopic to $\dot{\gamma}(t)/|\dot\gamma(t)|$, and so   it has degree $\pm 1$ as well. Since $N$ and $\tilde n|_{\Gamma}$ are also homotopic (by the same argument as before), we have that $\tilde n|_\Gamma$ has non-zero degree and hence   is surjective.  \end{proof}

\begin{remark}\label{rmk:surj}
{\rm The   proof of Part (i) of Theorem \ref{prop:surjectivity} shows that if $m$ is odd, then any continuous pseudonormal field is surjective, provided that  {\it a connected component of the boundary} has non-zero Euler characteristic. This is due to the fact that for $m$ odd  the degree of the Gauss map of a closed, smooth $(m-1)$-dimensional hypersurface in $\R^m$ is half of its Euler characteristic  (see for example \cite[p. 196]{guillemin}).
}
\end{remark}

\smallskip We   continue by investigating properties of the {\it multivalued} map of {\it all} pseudonormals. For  $\Omega\subset\R^m$  a bounded domain of class $C^0$ and $P\subset\partial\Omega$ we  let
$$\mcG (P)\defeq \{n\in S^{m-1}: n\mbox { is a  pseudonormal at  }P\},$$
and for any $E\subset\partial\Omega$ let $\mcG(E)\defeq \cup_{P\in E} \mcG(P)$. We denote by $\mathcal{P}(S^{m-1})$  the set of all subsets of $S^{m-1}$ and begin by noting that the  map $\mcG:\partial\Omega\rightarrow \mathcal{P}(S^{m-1})$ is lower semicontinuous. Indeed, by  definition this means that given any $P\in \partial\Omega$ and $n\in\mcG(P)$, for any neighbourhood $V$ of $n$ in $S^{m-1}$ there is a neighbourhood $U$ of $P$ in $\partial\Omega$ such that $\mcG(Q)\cap V\neq\emptyset$ for all $Q\in U$; this is obvious since $n\in \mcG(Q)$ for all $Q$ in a neighbourhood of $P$.

\par We use the following topological fact.
\begin{lemma}\label{topfact} If $\Omega\subset\R^m$ is a connected open set and $\mcU$ is a connected component of $\R^m\setminus\overline\Omega$ then $\partial\mcU$ is connected.
\label{lemma:topcon}
\end{lemma}
\begin{proof} This is a consequence of  \cite[49.VI,Theorem 2 and 57.I.9(i),57.III.1]{kuratowski}    (also noted in  \cite[Lemmas 4(i), 5]{kristensen}). \end{proof}
 In addition we have the following structural result.
\begin{lemma}
\label{comps}
Let $\Omega\subset\R^m$ be a bounded domain of class $C^0$. Then $\R^m\setminus \bar\Omega$ has a single unbounded connected component $\mathcal{D}$ and finitely many bounded connected components
$\mcU_i, \,i=1,\dots,k$,  each of which  is a bounded domain of class $C^0$ with $\R^m\setminus\bar\mcU_i$ connected. Furthermore $\partial\Omega$ can be written as the disjoint union 
\begin{equation}\label{disjointunion}
\partial\Omega=\partial\mathcal{D}\cup   \partial\mcU_1 \cup \cdots \cup \partial\mcU_k
\end{equation}
and the connected components of $\partial\Omega$ are the sets $\partial\mathcal{D}, \partial\mcU_i,\,i=1,\dots,k$.
\end{lemma}
\begin{proof}
\par  Since $\Omega$ is bounded, $\R^m\setminus \bar\Omega$ has a single unbounded connected component $\mathcal{D}$. If there were infinitely many bounded connected components $\mcU_i$ then we would have $x_i\to x\in\partial\Omega$  for some sequence with $x_i\in\mcU_i$,  which is easily seen to contradict that $\Omega$ is of class $C^0$. Since $\Omega$ is of class $C^0$ we have that $\partial\Omega=\partial(\R^m\setminus\bar\Omega)$, from which \eqref{disjointunion} follows.  The fact that all the sets in the union are disjoint  follows easily from   $\Omega$ being of class $C^0$.  Also 
$\R^m\setminus\overline\mcU_i=\cup_ {j\not=i}\overline{\mcU_j}\cup\overline{\mathcal{D}}\cup\Omega$, which is connected because all the sets in the union are connected and, for example, each point of $\partial\mathcal{D}$ (resp. $\partial\mcU_j, j\neq i$) has a neighbourhood consisting of points in $\bar{\mathcal{D}}\cup\Omega$ (resp. $\bar{\mcU_j}\cup\Omega$). Finally,
 by Lemma~\ref{lemma:topcon} each of the disjoint compact sets in \eqref{disjointunion} is connected, and so they are the connected components of $\partial\Omega$.
\end{proof}

We can now provide some properties of the image of $\mcG$:

\begin{proposition}\label{Gprops} Let $\Omega\subset\R^m, m\geq 2$ be a bounded domain of class $C^0$. Let $\mcC$ be a connected component of $\partial\Omega$. Then
    
{\rm (i)} {\rm Span}$\,\mcG(\mcC)=\R^m$,

{\rm (ii)} $\mcG(\mcC)$ is connected.

\end{proposition}
\begin{proof}
(i) We assume for contradiction that this is not true. Then Span$\,\mcG(\mcC)$ is contained in an $(m-1)$-dimensional affine subspace of $\R^m$ and thus $\mcG(\mcC)\subset \mathbb{S}^{m-1}\cap \{z\in\R^m:z\cdot N\geq 0\}$ for some $N\in \mathbb{S}^{m-1}$. By Lemma~\ref{comps} either $\mcC=\partial \mathcal D$   or $\mcC=\partial \mathcal U_i$ for some $i$. If $\mcC=\partial \mathcal D$ then sliding  a hyperplane with normal $N$ from $x\cdot N=+\infty$  until it touches $\mcC$ for the first time at some $P$, we find a good direction at $P$ belonging to $\{z\in\R^m:z\cdot N<0\}$, a contradiction. Similarly, if $\mcC=\partial \mathcal U_i$ for some $i$, then sliding such a hyperplane from $x\cdot N\to-\infty$ until it touches $\mcC$ for the first time, and recalling that by Lemma~\ref{comps} $\mcU_i$ is of class $C^0$, gives a similar contradiction. \\ \\
(ii) We assume for contradiction  that $\mcG(\mcC)$ is not connected.  Then $\mcG(\mcC)$ can be decomposed as $\mcG(\mcC)=A\cup B$ where $A,B$ are nonempty  sets and $A\cap \bar B=\bar A\cap B=\emptyset$.  Let $P\in \mcC$. We claim that either $\mcG(P)\subset A$ or $\mcG(P)\subset B$. Indeed if $n_1\in A, n_2\in B$ where $n_1,n_2\in\mcG(P)$ then since $\mcG(P)$ is convex
$$ n(\lambda)\defeq\frac{\lambda n_1+(1-\lambda)n_2}{|\lambda n_1+(1-\lambda)n_2|}\in\mcG(P) \mbox{ for all }\lambda\in [0,1].$$
But the sets $\{\lambda\in[0,1]:n(\lambda)\in A\}$ and $\{\lambda\in[0,1]:n(\lambda)\in B\}$ are relatively open and their union is $[0,1]$, contradicting the connectedness of $[0,1]$.

Now consider the sets $\mcC_A=\{P\in\mcC:\mcG(P)\subset A\}$ and $\mcC_B=\{P\in\mcC:\mcG(P)\subset B\}$, whose disjoint union is $\mcC$. Since $\mcC$ is connected, one of the sets $\mcC_A\cap \bar \mcC_B$, $\mcC_B\cap\bar\mcC_A$ is nonempty. Suppose, for example, that $P\in\mcC_A\cap \bar\mcC_B$. Then there exists a sequence $P_j\rightarrow P$ with $\mcG(P_j)\subset B$. Let $n\in\mcG(P)$. Then $n\in A$ but also $n\in\mcG(P_j)$ for sufficiently large $j$, and hence $n\in B$, a contradiction. \end{proof}

\section{Partial regularity of bounded $C^0$ domains}
In this section we show that if $\Omega\subset\R^m$ is a  {\it bounded} domain of class $C^0$ then\footnote{if $m\geq 4$ under the possibly unnecessary assumption that the Euler characteristic of $\Omega$ is nonzero} $\Omega$ has a Lipschitz boundary portion.
\begin{lemma}
 Let $\Omega\subset\R^m$, $m\ge 2$, be a bounded  domain of class $C^0$. If the set of good directions $($pseudonormals$)$ at a point $P\in\partial\Omega$ contains $m$ linearly independent directions then $\partial\Omega$ is Lipschitz in a neighbourhood of $P$, that is  for some $\delta>0$ and a suitable orthonormal coordinate system $Y\stackrel{\rm def}{=}(y',y_m)$ with origin at $P$
 \be
 \Omega\cap B(P,\delta)=\{y\in\R^m: y_m>f(y'),\,|y|<\delta\}
\label{liprep}
\ee
  where $f:\R^{m-1}\to\R$ is Lipschitz.
 \label{lemma:lipreg}
\end{lemma}

\begin{proof} Let $\{n_1,\dots,n_m\}$ be a set of linearly independent good directions at $P$ and let $\tilde n$ be an interior point of the geodesically convex hull  $\rm{co}_g\{n_1,\dots,n_m\}$ of $\{n_1,\dots,n_m\}$, which by Lemma~\ref{lemma:multiple} is also an interior point of the set of good directions at $P$. For example we can take

$$\tilde n=\frac{\sum_{i=1}^m n_i}{|\sum_{i=1}^m n_i|}.$$

\par Choosing an orthonormal coordinate system $Y=(y',y_m)$ with origin at $P$ and $\tilde n=e_m$ we have by Lemma~\ref{lemma:multiple} that the representation (\ref{liprep}) holds for some $\delta>0$ with $f:\R^{m-1}\to\R$ continuous. We claim that $f$ is Lipschitz in a neighbourhood of $0$. Then extending $f$ outside this neighbourhood to a Lipschitz map on $\R^{m-1}$, and choosing $\delta$ smaller if necessary, gives  the result.
\par If the claim were false there would exist sequences $S_j\to 0$, $T_j\to 0$ in $\R^{m-1}$ such that $S_j\not=T_j$ and
$$\frac{f(S_j)-f(T_j)}{|S_j-T_j|}\to\infty\,\textrm{ as }j\to\infty.$$
\par But for large enough $j$ the unit vector $$N_j\stackrel{\rm def}{=}\frac{(S_j,f(S_j))-(T_j,f(T_j))}{\sqrt{|S_j-T_j|^2+|f(S_j)-f(T_j)|^2}}$$ is not a good direction, since the line $t\to (T_j, f(T_j))+tN_j$ meets $\partial\Omega$ twice in $B(P,\delta)$ at $(T_j, f(T_j))$ and $(S_j,f(S_j))$. But $\lim_{j\to\infty}N_j=e_m$, contradicting that $\tilde n=e_m$ is an interior point of the set of good directions at $P$. \end{proof}

We continue by showing that if part of the boundary of a 2D domain is the graph of a nowhere differentiable function then there is exactly one good direction for each point on that part of the boundary.

\begin{lemma}\label{2Dex}
Let $f:(a,b)\to\R$ be a continuous nowhere differentiable function. Let $$\mathcal{G}\defeq\{(x,f(x))\in\R^2: x\in (a,b)\}$$ 
be a subset of the boundary $\partial\Omega$ of a bounded domain $\Omega\subset\R^2$ of class $C^0$.Then  at any point $P\in\mathcal{G}$ there exists a unique good direction namely $\pm(0,1)$ $($the sign being independent of $P\in\mathcal{G})$.
\end{lemma} 

\begin{proof} 
We can assume without loss of generality that $(0,1)$ is a good direction for all $P\in\mathcal G$, i.e. that $\Omega$ lies locally above $\mathcal G$.
Let us assume for contradiction that there exists a point $P\in\mathcal{G}$ with another good direction $n\in \mathbb{S}^1$ with $n\not=(0,1)$. Then there exists a whole neighbourhood of $P$ in $\mathcal{G}$ with two good directions $(0,1)$ and $n$. Restricting the interval $(a,b)$ if necessary we may thus assume   that for any $P\in\mathcal{G}$ there are two good directions $(0,1)$ and $n$.

 Then for any $m=(m_1,m_2)\in\mathbb{S}^1$ in the interior of the geodesically convex hull of $(0,1)$ and $n$ we have by Lemma~\ref{lemma:lipreg} that 
  $\mathcal{G}$ is the graph, along $m$, of a Lipschitz function, say $g$.  Then $g$ is differentiable almost everywhere, so that for almost all points in $\mathcal{G}$ there exists a tangent to $\mathcal{G}$.
 Let 
$$\mathcal{G}^*\defeq\{ P\in\mathcal{G}: \textrm{ the tangent to }\mathcal{G}\textrm{ at P exists and is parallel to  } (0,1)\}.$$
 We claim now that if the function $g$ is differentiable at the point $P\in\mathcal{G}\setminus\mathcal{G}^*$ then so is $f$. More precisely let us denote $P=(\bar x^*, \bar y^*)$ in the system of coordinates with axis $O\bar y=m$ and $O\bar x=m^\perp=(m_2,-m_1)$ respectively $P=(x^*,y^*)$ in the system of coordinates with axis $Oy=(0,1)$ and $Ox=(1,0)$.
 That $g$ is differentiable  means that 
 \be\label{diff:gx*}
 \frac{g(\bar x)-g(\bar x^*)}{\bar x-\bar x^*}\to g'(\bar x^*),\textrm{ as }\bar x\to\bar x^*.
 \ee
 \begin{figure}[h] \centering \def\svgwidth{300pt} 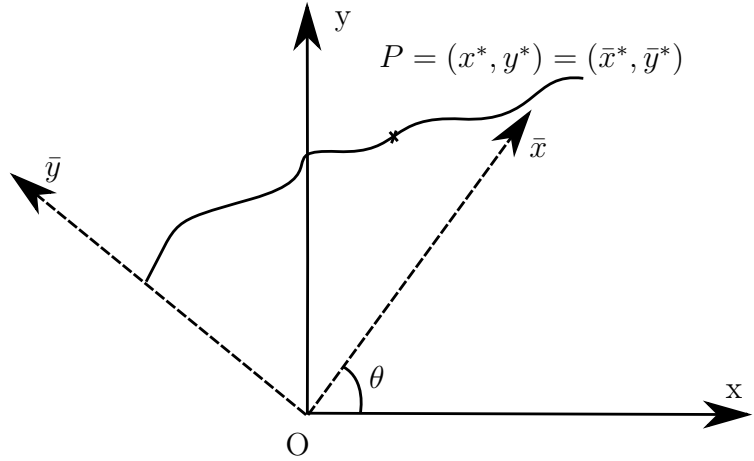
\caption{Changing the coordinate system}
\label{figure:coordinates}
\end{figure}   
 Let $\theta$ be the angle between the $\bar x$ axis and the $x$ axis (see Figure~\ref{figure:coordinates}). Letting $R(\theta)\defeq \left(\begin{array}{ll} \cos\theta  &\sin\theta \\ -\sin\theta & \cos\theta\end{array}\right)$ we have that $\left(\begin{array}{l} \bar x\\  \bar y\end{array}\right)=R(\theta)\left(\begin{array}{l}  x\\   y\end{array}\right)$, which implies that 
 \be\label{coord:initial}
 \begin{array}{rl} \bar x-\bar x^* &=\cos\theta\,(x-x^*)+\sin\theta\,[f(x)-f(x^*)]\\ g(\bar x)-g(\bar x^*) &=-\sin\theta\,(x-x^*)+\cos\theta\,[f(x)-f(x^*)]
 \end{array}
 \ee
 Hence
 \be 
 \frac{f(x)-f(x^*)}{x-x^*} 
 =\frac{\sin\theta\,+\cos\theta\,\left(\frac{g(\bar x)-g(\bar x^*)}{\bar x-\bar x^*}\right)}{\cos\theta\,-\sin\theta\,\left(\frac{g(\bar x)-g(\bar x^*)}{\bar x-\bar x^*}\right)}.
 \label{diff:f}
 \ee
 Note now that \eqref{coord:initial} and the continuity of $f$ imply that $\bar x\to \bar x^*$ when $x\to x^*$, which together with \eqref{diff:gx*} and \eqref{diff:f} implies
 $$\frac{f(x)-f(x^*)}{x-x^*}\to \frac{\sin\theta+\cos\theta\, g'(\bar x^*)}{\cos\theta-\sin\theta\, g'(\bar x^*)}\textrm{  as } x\to x^*$$ (note that $\cos\theta-\sin\theta g'(\bar x^*)\not=0$ by our assumption that $P=(\bar x^*,g(\bar x^*))\not\in\mathcal{G}^*$).  Hence $f$ is differentiable at $P$ as claimed. 
 
 There are now two cases:
\smallskip

{\bf Case I: } The measure of $\mathcal{G}^*$ is strictly smaller than the measure of $\mathcal{G}$. Then at all points $P=(x,f(x))\in\mathcal{G}\setminus\mathcal{G}^*$  we have that $f$ is differentiable at $x$. Thus we obtain a contradiction with the assumption that $f$ is nowhere differentiable.

\smallskip
{\bf Case II:} The measure of $\mathcal{G}^*$ is the same as that $\mathcal{G}$, hence $g'$ is  almost everywhere $\cot \theta$.    Then $\mathcal{G}$ is contained in a straight line in the direction $(0,1)$, and hence   is not a graph in the direction $(0,1)$. Thus we again obtain a contradiction. \end{proof}

\medskip
We now show that under a topological condition bounded domains with continuous boundary necessarily have Lipschitz boundary portions.

\begin{theorem}\label{thm:zeroEulerchar} Let $\Omega$ be a bounded domain of class $C^0$ in $\R^m$, $m\ge 3$. If $\Omega$ has non-zero Euler characteristic then  there exists a point $P\in\partial\Omega$ in the neighbourhood of which $\partial\Omega$ is Lipschitz.
\end{theorem}

\begin{proof}
We consider the canonical field of good directions obtained  in the proof of Proposition~\ref{prop:G}, namely
\be\label{def:Ccanonical}
G(P)\stackrel{\rm def}{=}\frac{\Sigma_{i=1}^k \alpha_i(P)n_{P_i}}{|\Sigma_{i=1}^k \alpha_i(P)n_{P_i}|}
\ee
 with $n_{P_i}$ a good direction at $P_i$ with respect to the ball $B(P_i,\delta_i)$ and thus at each point $R\in B(P_i,\frac{1}{2}\delta_{P_i}),\,i=1,\dots,k,$ and $\alpha_i, i=1,\dots,k,$ a partition of unity subordinate to the covering $B(P_i,\frac{1}{2}\delta_{P_i}),\,\,i=1,\dots,k$.  Then $G|_{\partial\Omega}$ is a continuous pseudonormal  field and hence Theorem~\ref{prop:surjectivity}, part (i), provides that $G:\partial\Omega\to\Sph^{m-1}$ is surjective. We claim now that there exist $m$ linearly independent good directions $\bar n_1,\dots,\bar n_m\in\{n_{P_i},\,i=1,\dots,k\}$ such that $\partial\Omega$ and the $m$ corresponding balls, on which they are good directions, have a non-empty intersection.  Lemma~\ref{lemma:lipreg} then shows that $\partial\Omega$ is Lipschitz in a neighbourhood of any point $P$ in their intersection.

  In order to prove the claim we assume, for contradiction, that  at each point $Q\in\partial\Omega$ the subset of  $\{n_{P_i}\in\Sph^{m-1},i=1,\dots,k\}$ that are good directions at $Q$ is contained in a hyperplane. For each $i=1,\dots,k$, let $E_i=\{P\in\partial\Omega:\alpha_i(P)>0\}$. Then each $\partial E_i$ is a closed nowhere dense subset of $\partial\Omega$, and so by the Baire Category theorem $\cup_{i=1}^k\partial E_i$ is a closed nowhere dense subset of $\partial\Omega$. Consider any nonempty subset of the form $A_{i_1,\ldots,i_r}=\cap_{j=1}^rE_{i_j}$, where $1\leq r\leq k$ and  $1\leq i_1<\cdots <i_r\leq k$. For $P\in A_{i_1,\ldots,i_r}$ we have by our assumption that the vectors $n_{P_{i_j}},\;1\leq j\leq r,$ lie in a hyperplane, and thus from \eqref{def:Ccanonical} we have that $G(A_{i_1,\ldots,i_r})$ is contained in this hyperplane. Hence $G(\partial\Omega\setminus\cup_{i=1}^k\partial E_i)$ is contained in a closed nowhere dense subset $A$ of $\Sph^{m-1}$ (the union of the intersection with  $\Sph^{m-1}$ of a finite number of hyperplanes).  Since $G$ is continuous and   $  \cup_{i=1}^k\partial E_i$  is nowhere dense, it follows that $G(\partial\Omega)\subset A$, contradicting the surjectivity of $G$.
 \end{proof}

\medskip\par We continue by studying  the set of all pseudonormals in the general case of $C^0$ domains with no topological restrictions imposed on $\partial\Omega$.

\begin{lemma}\label{lemma:nonsingleton} Let $\Omega$ be a bounded domain of class $C^0$ in $\R^m$, $m>1$. For each connected component $\mathcal{C}$ of $\partial\Omega$ there exists a point $P\in\mathcal{C}$ at which the set of good directions at $P$ is not a singleton.
\end{lemma}
\begin{proof} Assume for contradiction that for any $P\in\mathcal{C}$ there exists only one good direction $n(P)\in\Sph^{m-1}$. Note that if $N$ is a good direction at some point $x\in\partial\Omega$ then $N$ is also a good direction at points $z\in\partial\Omega$ sufficiently close to $x$. Pick $P_1\in\mathcal{C}$ and define $E=\{P\in\mathcal{C}: n(P)=n(P_1)\}$. The preceding property implies that $E$ is both open and closed in $\mathcal{C}$. Since $\mathcal{C}$ is connected it follows that $E=\mathcal{C}$, in contradiction to Proposition~\ref{Gprops}(i). \end{proof}

\smallskip The   last lemma, combined with the characterization of Lipschitz regularity of the boundary in Lemma~\ref{lemma:lipreg} immediately imply:

\begin{theorem}\label{partial2Da} Let $\Omega$ be a bounded domain of class $C^0$ in $\R^2$. For each connected component $\mathcal{C}$ of $\partial\Omega$ there exists a point $P\in\mathcal{C}$ in the neighbourhood of which $\partial\Omega$ is Lipschitz.
\end{theorem}
\begin{remark}\label{unbounded}{\rm
Note that Theorem \ref{partial2Da} is not in general true for {\it unbounded} domains in $\R^2$ of class $C^0$. Indeed, by Lemma \ref{2Dex}, if $f:\R\rightarrow\R$ is nowhere differentiable then the domain $\Omega=\{(x_1,x_2):x_2>f(x_1)\}$ has no Lipschitz boundary portions.}
\end{remark}

We also have a similar result in 3D, but the proof is considerably more intricate.

\begin{theorem}\label{partial3D}Let $\Omega$ be a bounded domain of class $C^0$ in $\R^3$. For each connected component of $\partial\Omega$ there exists a point $P\in\partial\Omega$ in the neighbourhood of which $\partial\Omega$ is Lipschitz.
\end{theorem}

\begin{proof} We consider again the canonical field of good directions obtained  in the proof of Proposition~\ref{prop:G}, namely
\be\label{def:Ccanonical1}
G(P)\stackrel{\rm def}{=}\frac{\sum_{i=1}^k \alpha_i(P)n_{P_i}}{|\sum_{i=1}^k \alpha_i(P)n_{P_i}|}
\ee
 with $n_{P_i}$ a good direction at $P_i$ with respect to the ball $B(P_i,\delta_i)$ and thus at each point $R\in B(P_i,\frac{1}{2}\delta_{P_i}),\,i=1,\dots,k,$ and $\alpha_i, i=1,\dots,k,$ a partition of unity subordinate to the covering $B(P_i,\frac{1}{2}\delta_{P_i}),\,\,i=1,\dots,k$. 

We continue by analyzing the image of $G$ when restricted to an arbitrary connected component  ${\mathcal S}={\mathcal S}_\varepsilon$ of  $\partial\Omega_\varepsilon$  for the approximating $\Omega_\varepsilon$ as in Theorem~\ref{homapprox}, with $\varepsilon>0$ sufficiently small. We will show that $G(\mcS)$ has non-empty interior and that this allows one to infer that there exist three linearly independent directions at some point $Q_\eps\in \partial\Omega_\eps$ for all $\eps>0$. Then, thanks to the special structure of the canonical field $G$, it will be shown that  the same can be claimed at some point $Q\in\partial\Omega$.

\smallskip
We start by noting that if the connected component $\mcS$ of  $\partial\Omega_\varepsilon$ has non-zero Euler characteristic then Theorem~\ref{prop:surjectivity} and Remark~\ref{rmk:surj}   ensure that $G(\mcS)=\Sph^2$ hence $G(\mcS)$ has non-empty interior.

The case when $\mcS$ has zero Euler characteristic is more delicate as in this situation the field $G|_{\mcS}$ is not necessarily surjective, as was shown in Proposition~\ref{prop:goodtorus}. The proof continues in two steps.    In the first we show  that we can assume, without loss of generality, that $\mcS=\To$ is the standard torus $\To=\Sph^1\times\Sph^1$ embedded in $\R^3$, for which we know by Proposition~\ref{prop:goodtorus} that
any  smooth field of good directions $G:\mcS\to\Sph^2$ has an image $G(\mcS)$ with non-empty interior. Then, in the second step, we show that $G(\mcS)$ having non-empty interior implies (irrespective of the Euler characteristic of $\mcS$ being zero or not)  that there exist three linearly independent directions at some point $Q_\eps\in \partial\Omega_\eps$ for all $\eps>0$, and that the same holds also at some point $Q\in\partial\Omega$.

{\bf Step 1 (reduction to the standard torus and consequences):}   We note that $\mcS$ is a smooth $2$-dimensional compact, connected and orientable manifold without boundary, that   has zero Euler characteristic. Let 
\be\label{torus:explicit}
\To\defeq \{(\cos\theta\,(2+\cos\varphi),\sin\theta\,(2+\cos\varphi), \sin\varphi): \theta,\varphi\in [0,2\pi]\}
\ee denote a standard torus embedded in $\R^3$. Then $\To$ is a $2$-dimensional compact, connected and orientable manifold without boundary of zero Euler characteristic and by the theorem of classification of $2$-dimensional compact manifolds \cite[Chapter 9]{hirsch} there exists a diffeomorphism  $D:\mcS\to\To$.

 We  use the diffeomorphism $D$ to transport the field of good directions, in a bijective manner, from $\mcS$ onto $\To$, while preserving the angle between the good direction and the normal. 
 \begin{lemma}\label{lemma:paralleltransport}
 Let $\nu_S:\mcS\to \Sph^2$, $\nu_{\To}:\To\to\Sph^2$ denote the interior normals to $\mcS$, respectively $\To$. There exist smooth functions $e,\hat e:\mcS\to\Sph^2$, $f,\hat f:\To\to\Sph^2$  such that  $e,\hat e$ and $\nu_{\mcS}$, respectively $f,\hat f$ and $\nu_{\To}$ are pairwise orthogonal at each point.
 \end{lemma} 
\begin{proof}
The proof is essentially an easy consequence of the fact that $\mcS$ and $\To$ are parallelizable manifolds embedded in $\R^3$. We consider the torus $\To$ as in \eqref{torus:explicit}. Then $\nu_{\To}(\theta,\varphi)=-(\cos\theta\cos\varphi,\sin\theta\cos\varphi, \sin\varphi)$ for any $\theta,\varphi\in [0,2\pi]\times [0,2\pi]$.  We take 
 
 $$f(\theta,\varphi)\defeq (-\sin\theta,\cos\theta,0),\quad \hat f(\theta,\varphi)\defeq (-\cos\theta\sin\varphi,-\sin\theta\sin\varphi,\cos\varphi).$$
 
 We consider the derivative $TD:T\mcS\to\To$ of the diffeomorphism $D:\mcS\to\To$, acting between the tangent bundles $T\mcS$ and $T\To$. Then for any point $P\in\mcS$ we have that the linear map $T_PD: T_P\mcS\to T_{D(P)}\To$ is an invertible linear function, and as such $T_PD^{-1}\big(f(D(P))\big)$ and $T_PD^{-1}\big(\hat f(D(P))\big)$ define a basis in  $T_P\mcS$ which varies smoothly in $P$. However this basis need not be an orthogonal basis and in order to obtain an orthogonal, smoothly varying, basis, we take:
 $$e(P)\defeq \frac{T_PD^{-1}\big(f(D(P))\big)}{|T_PD^{-1}\big(f(D(P))\big)|},\quad \hat e(P)=e(P)\times \nu_\mcS (P).$$ 
\end{proof} 
  Continuing the proof of Theorem \ref{partial3D}, let $G:\mcS\to\Sph^2$ be a smooth field of good directions on $\mcS$ such that $G(P)\cdot \nu_\mcS (P)>0$ for all $P\in\mcS$. We define
$$
\tilde G(D(P))\stackrel{\rm def}{=} \big(G(P)\cdot\nu_{\mcS}(P)\big)\nu_{\To}(D(P))+\big(G(P)\cdot e(P)\big)f(D(P))+\big( G(P)\cdot \hat e(P)\big)\hat f(D(P))
$$
 We have then that $\tilde G(Q)\cdot \nu_{\To}(Q)=G(D^{-1}(Q))\cdot \nu_{\mcS}(D^{-1}(Q))>0$ for all $Q\in\To$ and thus $\tilde G$ is a smooth field of good directions on $\To$, the transported version to $\To$ of the field $G$.
 
 We claim now that $G(\mcS)$ has nonempty interior if and only if the transported version $\tilde G(\To)$ has nonempty interior.
 
To this end let us define the continuous function $\mathcal{H}:\mcS\times \Sph^2\to\To\times\Sph^2$ by
$$\mathcal{H}(P,n)\stackrel{\rm def}{=}\Big(D,\big(n\cdot\nu_{\mcS}\big)\nu_{\To}(D)+\big(n\cdot e\big)f(D)+\big( n\cdot \hat e\big)\hat f(D)\Big)(P)$$
One can check that $\widetilde{\mathcal{H}}:\To\times \Sph^2\to\mcS\times\Sph^2$ defined by
$$\widetilde{\mathcal{H}}(R,m)\stackrel{\rm def}{=}\Big(D^{-1},\big(m\cdot \nu_{\To}\big)\nu_{\To}(D^{-1})+\big(m\cdot f\big)e(D^{-1})+\big(m\cdot \hat f)\hat e(D^{-1})\Big)(R)
 $$ is the continuous inverse of $\mathcal{H}$ and thus $\mathcal{H}$ is a homeomorphism. Moreover we have:
 \be\label{rel:transportfield}
 \mathcal{H}(P,G(P))=(D(P),\tilde G(D(P))).
 \ee
 Let us assume now that $G(\mcS)$ has nonempty interior $E$. Then,  since $G$ is continuous, $G^{-1}(E)$ is nonempty and open. Since $\mathcal{H}$ is a homeomorphism, it takes nonempty open sets into nonempty open sets, so that, by \eqref{rel:transportfield},  $\mathcal{H}(G^{-1}(E),E){=}(D(G^{-1}(E),\tilde G(D(G^{-1}(E))))$. Thus $\tilde G(D(G^{-1})(E))\subset\tilde G(\To)$ is a nonempty open set and therefore it has nonzero measure.  One can argue in a similar way, using $\tilde H$ to show that if $\tilde G(\To)$ has nonempty interior then so does $G(\mcS)$, thus proving our claim.

Proposition~\ref{prop:goodtorus}  now shows that $\tilde G(\To)$ has  nonempty interior and therefore  so does $G(\mcS)$.

{\bf Step $2$ (from the smooth approximating domains back to the rough one):} Let $\Omega_\eps$ be a sequence of smooth domains approximating $\Omega$, as given in Theorem~\ref{homapprox}. Let $G:V\to\Sph^2$ be a canonical field of good directions, where $V$ is an open set containing $\partial\Omega$ (hence there exists an $\eps_0>0$ so that $\partial\Omega_\eps\subset V$ for $0<\eps<\eps_0$). The previous step and the remark before it show that for each connected component $\mcS_\varepsilon$ of $\partial\Omega_\varepsilon$ we have that the interior of $G(\mcS_\varepsilon)$ is nonempty.

 Let us recall now the definition \eqref{def:Ccanonical1} of the canonical field $G$. Arguing as in the proof of Proposition~\ref{thm:zeroEulerchar} (and using  the fact that $G(\mcS_\varepsilon)$ has nonempty interior, instead  of surjectivity of $G$) we have that there exists a point $Q_\eps\in\mcS_\varepsilon$ so that there  exist three linearly independent good directions $n_{P_{i_1(\varepsilon)}},n_{P_{i_2(\varepsilon)}},n_{P_{i_3(\varepsilon)}}$      such that the three corresponding balls $B(P_{i_k(\varepsilon)},\frac{1}{2}\delta_{i_k(\varepsilon)}), k=1,2,3,$ have a nonempty intersection containing $Q_\eps$.  

 Since the number of balls in the cover is finite there exist three of them, say $B(\bar P_i,\frac{\bar\delta_i}{2}),i=1,2,3$ such that there exist infinitely many points $Q_{\eps_j}$, $\varepsilon_j\rightarrow 0+$,  in their intersection. There exists then a $\bar Q\in\partial\Omega$, a limit point of the $Q_{\eps_j}$, with $\bar Q\in \cap_{i=1}^3 \overline{B(\bar P_i,\frac{\delta_i}{2})}\subset \cap_{i=1}^3B(\bar P_i,\delta_i)$ and such that there are three linearly independent good directions at $\bar Q$.   \end{proof}

\begin{remark}{\rm Despite some efforts we have been unable to decide whether Theorem \ref{thm:zeroEulerchar} remains true for $m\geq 4$  when the Euler characteristic of $\Omega$ is zero.}
\end{remark}
\begin{remark}\label{gendomain}{\rm
Theorems \ref{thm:zeroEulerchar}-\ref{partial3D} remain valid if the hypothesis that $\Omega$ is of class $C^0$ is replaced by the weaker hypothesis that $\Omega$ is the  image under a $C^1$ diffeomorphism $\varphi:U\rightarrow\R^m$ of a bounded domain $\Omega'\subset\R^m$ of class $C^0$, where $U$ is an open neighbourhood of $\overline{\Omega'}$. Indeed, by \cite[Theorem 4.1]{hofmann} (in which the term strongly Lipschitz is used instead of our Lipschitz) if $\partial\Omega'$ is Lipschitz in a neighbourhood of $P$ then $\partial\Omega$ is Lipschitz in a neighbourhood of $\varphi(P)$.}
\end{remark}

\begin{remark}\label{partreg:topman}{\rm
The partial regularity results of this section  crucially use a rigidity specific to the definition of $C^0$ domains. Indeed, just by enlarging a bit the class of domains studied, to topological manifolds with boundaries, one can immediately provide  counterexamples to partial regularity. This  can be seen for instance by considering the domain in polar coordinates in $\R^2$: $\Omega\defeq \{ (r,\theta); 0\le r<f(\theta)\}$ with $f:\mathbb{S}^1\to\R$ a nowhere differentiable function, which is a topological manifold with boundary but is not a domain of class $C^0$.
}
\end{remark}

\section*{Acknowledgements} The research of both authors was partly supported by   EPSRC grants EP/E010288/1 and by the EPSRC Science and Innovation award to the Oxford Centre for Nonlinear
PDE (EP/E035027/1). The research of JMB was also supported by the European Research Council under the European Union's Seventh Framework Programme
(FP7/2007-2013) / ERC grant agreement no 291053 and
 by a Royal Society Wolfson Research Merit Award.  The research of AZ  was partially supported by the Basque Government through the BERC 2014-2017 program; and by the Spanish Ministry of Economy and Competitiveness MINECO: BCAM Severo Ochoa accreditation SEV-2013-0323.

We are especially grateful to Marc Lackenby for  the suggestion of an example forming the basis of the second part of Proposition~\ref{prop:goodtorus} as well  as for other advice,   to Vladimir \v{S}ver\'{a}k, whose perceptive questions after a seminar of JMB at the University of Minnesota on an earlier version of the paper led to radical improvements, and to Rob Kirby for long discussions concerning smoothing of manifolds that are incorporated in Remark \ref{topman1}. We are also grateful to Moe Hirsch for his interest and various references, and to Nigel Hitchin for useful comments. 

  Finally we are grateful to  anonymous referees for pointing out to us the connection with the theory of smoothing of manifolds, for a simplification of our original example in the second part of Proposition~\ref{prop:goodtorus}, and for  mentioning various relevant references.

\end{document}